\documentclass[reqno,12pt,letterpaper]{amsart}
\usepackage{geometry}
\usepackage{ulem}
\usepackage{enumerate}
\usepackage{amssymb,amsthm,amsmath}
\usepackage[colorlinks,breaklinks]{hyperref}
\usepackage{tikz}
\usepackage{booktabs} 

\theoremstyle{plain}
\newtheorem{prop}{Proposition}[section]
\newtheorem{lemma}[prop]{Lemma}
\newtheorem{theorem}[prop]{Theorem}
\newtheorem{cor}[prop]{Corollary}

\theoremstyle{remark}
\newtheorem{remark}[prop]{Remark}

\newcommand{\Aa}{\textbf{a}}
\newcommand{\Bb}{\textbf{b}}
\newcommand{\Cc}{\textbf{c}}
\newcommand{\Dd}{\textbf{d}}
\newcommand{\Ee}{\textbf{e}}
\newcommand{\Ff}{\textbf{f}}
\newcommand{\R}{\mathbb{R}}
\newcommand{\HH}{\mathbb{H}}
\newcommand{\CC}{\mathbb{C}}
\newcommand{\mat}[4]{\begin{pmatrix} #1&#2\\#3&#4\end{pmatrix}}
\newcommand{\pmat}[4]{\begin{pmatrix} #1&#2\\#3&#4\end{pmatrix}}

\allowdisplaybreaks

\makeatletter
\@namedef{subjclassname@2020}{%
  \textup{2020} Mathematics Subject Classification}
\makeatother

\DeclareMathOperator{\dist}{dist}
\DeclareMathOperator{\dvol}{dvol}
\DeclareMathOperator{\vol}{vol}
   
\begin{document}

\title[]{Resonance counting and spectral gaps for hyperbolic surfaces via Fourier decay}
\author[L.\@ Soares]{Louis Soares}
\email{louis.soares@gmx.ch}

\begin{abstract}
We prove an improved fractal Weyl upper bound for resonances of infinite-area convex cocompact hyperbolic surfaces, expressed in terms of the Hausdorff dimension of the limit set and the Fourier decay exponent of the Patterson--Sullivan measure. This bound also yields a sharper upper bound for the essential spectral gap for these surfaces.
\end{abstract}
\maketitle

\section{Introduction}
Let $X = \Gamma \backslash \HH^2$ be a non-elementary convex cocompact hyperbolic surface of infinite area, and let $\Delta_X$ be its Laplace--Beltrami operator. After conjugating $\Gamma$, we may assume that its limit set $\Lambda$ is a compact subset of $\mathbb{R}$. It consists of the accumulation points of any orbit $\Gamma z$, $z \in \HH^2$, and is a perfect nowhere-dense set. We denote by $\delta \in (0,1)$ the Hausdorff dimension of $\Lambda$. The resonances of $X$ are the poles of the meromorphic continuation of
\[
R_X(s) := (\Delta_X - s(1-s))^{-1};
\]
see Borthwick's book \cite{Bor16} for more background. By the results of Patterson and Sullivan \cite{Pa76,Sul79}, $s=\delta$ is a simple resonance, and there are no other poles in \(\{\mathrm{Re}(s)\ge\delta\}\).

If $\delta > \frac{1}{2}$, then $X$ has only finitely many resonances in \(\{\mathrm{Re}(s)>1/2\}\). They lie in the interval \((\frac{1}{2},\delta]\) and are exactly those coming from eigenvalues of \(\Delta_X\) in \((0,1/4)\). Moreover, Bourgain--Dyatlov \cite{BD18} proved that there exists \(\beta = \beta(X) >0\) such that \(\{\mathrm{Re}(s)>\frac{1}{2}-\beta\}\) contains only finitely many resonances.

We now allow any \(\delta\in(0,1)\). Naud~\cite{Naud2} proved that there exists \(\epsilon_0>0\) (depending on \(X\)) such that \(\{\mathrm{Re}(s)>\delta-\epsilon_0\}\) contains no resonances except \(s=\delta\). Let \(g(X)\) denote the essential spectral gap, i.e. the infimum of \(\sigma\) such that \(\{\mathrm{Re}(s)>\sigma\}\) contains only finitely many resonances. Naud's result gives \(g(X)\le\delta-\epsilon_0\). Jakobson--Naud \cite{JN2012} conjectured that \(g(X)=\delta/2\). The heuristic for this conjecture is a square-root cancellation argument for the spectral radius of the classical transfer operator \(\mathcal{L}_s\) associated with Schottky groups. The operator $\mathcal{L}_s$ depends holomorphically on the spectral parameter $s\in \mathbb{C}$ and is a complexified version of the Ruelle transfer operator of the Bowen--Series map induced by the action of \(\Gamma\) on the limit set \(\Lambda\).

In \cite{BD17}, Bourgain--Dyatlov proved the following quantitative Fourier decay result for the Patterson--Sullivan probability measure:

\begin{theorem}[{\cite[Theorem~2]{BD17}}]\label{thm:BD_fourier_decay}
Let $X = \Gamma \backslash \mathbb{H}^2$ be a non-elementary infinite-area convex cocompact hyperbolic surface, and let $\delta \in (0,1)$ be the Hausdorff dimension of its limit set $\Lambda$. Let $\mu$ be the Patterson--Sullivan probability measure based at $\mathrm{i}\in\mathbb{H}^2$. Assume that
\[
\varphi \in C^2(\mathbb{R}; \mathbb{R}), \qquad g \in C^1(\mathbb{R}; \mathbb{C})
\]
are functions satisfying the following bounds for some constant $C_{\varphi,g}$:
\[
\|\varphi\|_{C^2} + \|g\|_{C^1} \le C_{\varphi,g}, \qquad \inf_{\Lambda} |\varphi'(x)| \ge C_{\varphi,g}^{-1}.
\]
Then there exist constants $\beta > 0$ depending only on $\delta$, and $C > 0$ depending only on $X$ and $C_{\varphi,g}$, such that
\begin{equation}\label{eq:BD_decay}
\left| \int_{\Lambda} e^{i\xi \varphi(x)} g(x) \, d\mu(x) \right| \le C |\xi|^{-\beta/2}
\qquad \text{for all } |\xi| > 1.
\end{equation}
\end{theorem}
From now on, we let $\beta_{\mathrm{F}}>0$ denote the best exponent $\beta$ for which Theorem~\ref{thm:BD_fourier_decay} holds, and we call it the Fourier decay exponent. In the same paper, Bourgain--Dyatlov then deduced from this result that 
\begin{equation}\label{eq:BD_spectral_gap}
g(X)\le g_{\mathrm{BD}}:=\delta-\frac{\beta_{\mathrm{F}}}{8},
\end{equation}
using the fractal uncertainty principle developed earlier by Dyatlov--Zahl \cite{DZ16}; see Theorem~1 and the discussion following Theorem~2 in \cite{BD17}. Note that Theorem~\ref{thm:BD_fourier_decay} implies that the Fourier dimension of $\Lambda$ is bounded below by $\beta_{\mathrm{F}}$. Since the Hausdorff dimension of a set is greater than or equal to its Fourier dimension, we have $\beta_{\mathrm{F}} \le \delta$.

A related question studied in the literature is the distribution of resonances. Let $\mathcal{R}_X$ be the set of resonances, counted with multiplicities, where the multiplicity of a resonance $s_0$ is the rank of the residue of $R_X(s)$ at $s_0$. For $\sigma\in\mathbb{R}$ and $T>0$, define the counting function
\[
M_X(\sigma,T) := \#\{\, s\in\mathcal{R}_X : \mathrm{Re}(s)\ge\sigma,\ |\mathrm{Im}(s)-T|\le 1 \,\}.
\]
Guillop\'e--Lin--Zworski~\cite{GLZ} proved the fractal Weyl bound $M_X(\sigma,T)=O(T^\delta)$ as $T\to\infty$. The fractal Weyl law conjecture~\cite{Lu_Sridhar_Zworski} predicts that this bound is sharp for sufficiently negative $\sigma$. Define $m_X(\sigma)$ as the infimum of exponents $\alpha$ such that $M_X(\sigma,T)=O(T^{\alpha+\epsilon})$ for all $\epsilon>0$ as $T\to\infty$. Building on the work of Naud~\cite{Naud14}, Jakobson--Naud \cite{JN2016} proved that $m_X(\sigma)\leq\delta-\tau(\sigma)$ for some function $\tau(\sigma)$ that is positive for $\sigma\in(\delta/2,\delta)$. Their result is stated for convex cocompact congruence subgroups of $\mathrm{SL}_2(\mathbb{Z})$, but their method applies to any $X$. This function can be expressed in terms of the topological pressure of the Bowen--Series map, but it is difficult to evaluate explicitly. Dyatlov~\cite[Theorem~1]{Dya19} proved the explicit bound
\begin{equation}\label{intro:dyatlov}
m_X(\sigma)\le 2(\delta-\sigma) = \delta - 2 \left(\sigma - \frac{\delta}{2}\right).
\end{equation}
In the appendix of~\cite{Dya19}, Borthwick--Dyatlov--Weich carried out numerical computations for several explicit convex cocompact surfaces. These computations suggest that this bound is not optimal, although it remains close to the observed counting exponents for the surfaces studied. Cunningham~\cite[Theorem~1.1]{Cun26} recently proved that
\begin{equation}\label{intro:cunningham}
m_X(\sigma)\le 4\left(\delta-\sigma\right) - \frac{\beta_{\mathrm{F}}}{2}.
\end{equation}
This bound has the nice property that it recovers the Bourgain--Dyatlov spectral gap, since it yields $m_X(\sigma)<0$ for all $\sigma>g_{\mathrm{BD}}$. We should mention that both \cite{Dya19} and \cite{Cun26} define resonances as poles of $(\Delta_X-\frac14-\lambda^2)^{-1}$, which is related to our parametrization by $s=\frac12-i\lambda$.

Our main theorem provides sharper explicit counting and spectral gap estimates. See also Tables~\ref{tab:counting_exponent_bounds} and \ref{tab:essential_spectral_gap} for a comparison with previous results.

\begin{theorem}[Main Theorem]\label{thm:main}
Let $X$ be a non-elementary infinite-area convex cocompact hyperbolic surface. Then, for all $\sigma\in(\delta/2,\delta]$, we have
\begin{equation}\label{eq:main_counting_bound}
m_X(\sigma)\le \left( 2 + \frac{\beta_{\mathrm{F}}}{\delta} \right)\left(\delta-\sigma\right) - \frac{\beta_{\mathrm{F}}}{2} = \delta-  \left(2+\frac{\beta_{\mathrm{F}}}{\delta}\right) \left(\sigma - \frac{\delta}{2}\right),
\end{equation}
where $\beta_{\mathrm{F}}$ is the Fourier decay exponent defined after Theorem~\ref{thm:BD_fourier_decay}. In particular, we have
\begin{equation}\label{eq:main_spectral_gap}
g(X)\le \delta - \frac{\delta \beta_{\mathrm{F}}}{4\delta + 2\beta_{\mathrm{F}}}.
\end{equation}
\end{theorem}

Since $0 < \beta_{\mathrm{F}}\le\delta$, we have $2 < 2 + \frac{\beta_{\mathrm{F}}}{\delta}\le3$. Therefore, Theorem~\ref{thm:main} simultaneously improves both previously known bounds \eqref{intro:dyatlov} and \eqref{intro:cunningham} for all $\sigma\in(\delta/2,\delta)$. The upper bound for $g(X)$ follows directly from the improved counting estimate, since the right-hand side of \eqref{eq:main_counting_bound} becomes negative whenever $\sigma$ is greater than the right-hand side of \eqref{eq:main_spectral_gap}. Indeed, for such $\sigma$ we get $M_X(\sigma,T)=0$ for all large $T$; since the set of resonances is discrete and symmetric under complex conjugation, this implies that only finitely many resonances lie in $\{\mathrm{Re}(s)\ge\sigma\}$.

\begin{table}[htbp]
\centering
\caption{Comparison of explicit bounds for the local counting exponent $m_X(\sigma)$ in the range $\sigma \le \delta$.}
\label{tab:counting_exponent_bounds}
\begin{tabular}{l|l}
\toprule
Source & Bound for $m_X(\sigma)$ \\ 
\midrule
Guillop\'e--Lin--Zworski~\cite{GLZ} & $\delta$ \\
Dyatlov~\cite{Dya19} & $2(\delta-\sigma)$ \\
Cunningham~\cite{Cun26} & $4(\delta-\sigma)-\beta_{\mathrm{F}}/2$ \\
Theorem~\ref{thm:main} & $\left( 2 + \frac{\beta_{\mathrm{F}}}{\delta} \right)\left(\delta-\sigma\right) - \frac{\beta_{\mathrm{F}}}{2}$ \\
\bottomrule
\end{tabular}
\end{table}

Combining Theorem \ref{thm:main} with the estimate of Guillop\'e--Lin--Zworski~\cite{GLZ} yields the following estimate valid for all $\sigma \le \delta$:
\[
m_X(\sigma) \le \min\left\{\delta, \left( 2 + \frac{\beta_{\mathrm{F}}}{\delta} \right)\left(\delta-\sigma\right) - \frac{\beta_{\mathrm{F}}}{2}\right\}.
\]

\begin{table}[htbp]
\caption{Comparison of explicit bounds for the essential spectral gap}
\label{tab:essential_spectral_gap}
\centering
\begin{tabular}{l|l}
\toprule
Source & Upper bound for $g(X)$ \\ 
\midrule
Naud~\cite{Naud2} & $\delta-\epsilon,$ where $\epsilon = \epsilon(X) > 0$ \\
Bourgain--Dyatlov \cite{BD18} & $\frac{1}{2}-\beta$, where $\beta=\beta(X)>0$ \\
Bourgain--Dyatlov \cite{BD17} & $\delta-\frac{\beta_{\mathrm{F}}}{8}$ \\
Theorem~\ref{thm:main} & $\delta - \frac{\delta\beta_{\mathrm{F}}}{4\delta+2\beta_{\mathrm{F}}}$ \\
\bottomrule
\end{tabular}
\end{table}

We note that Bourgain--Dyatlov's argument \cite{BD17} relies on Bourgain's sum-product estimates \cite{Bourgain2010} and does not provide an explicit lower bound for $\beta_{\mathrm{F}}$ in terms of $\delta$. Recently, Lequen--Sahlsten \cite{LequenSahlsten} proved an explicit Fourier decay estimate for the Patterson--Sullivan measure when $\delta > 1/2$, provided the phase is linear. In contrast, Theorem~\ref{thm:BD_fourier_decay} also applies to non-linear phases. In the optimal case where $\beta_{\mathrm{F}} = \delta$, our main theorem would lead to $g(X)\le \frac{5\delta}{6}$. Note that the equality $\beta_{\mathrm{F}} = \delta$ would imply that $\Lambda$ is a Salem set.

Finally, we compare our main theorem with the global counting function
\[
N_X(\sigma,T):=\#\{s\in\mathcal{R}_X:\mathrm{Re}(s)\ge\sigma,\,|\mathrm{Im}(s)|\le T\}.
\]
In \cite{Soa25}, the author proved that for $\sigma\in(\delta/2,\delta]$,
\begin{equation}\label{intro:Soares_global_counting}
N_X(\sigma,T) = O\Big(T^{1+\delta-  4 \left(\sigma - \frac{\delta}{2}\right)+\epsilon}\Big), \qquad T\to\infty.
\end{equation}
By integration over $T$, we have $N_X(\sigma,T) = O(1+T^{1+m_X(\sigma)+\epsilon})$. Hence, Theorem~\ref{thm:main} immediately gives the following upper bound for all $\sigma\in(\delta/2,\delta]$:
\[
N_X(\sigma,T) = O\Big(T^{1+\delta-  \left(2+\frac{\beta_{\mathrm{F}}}{\delta}\right) \left(\sigma - \frac{\delta}{2}\right)+\epsilon}\Big), \qquad T\to\infty.
\]
However, this does not improve \eqref{intro:Soares_global_counting}, because $2+\frac{\beta_{\mathrm{F}}}{\delta} < 4$.
On the other hand, the methods of \cite{Soa25} do not give improved bounds for $M_X$ and do not yield a spectral gap.

The paper is organized as follows. In Section~\ref{sec:Prelims}, we review basic material including Schottky groups, the associated transfer operators, thermodynamic formalism, and properties of Bergman kernels. In Section~\ref{sec:main_thm}, we prove the main theorem. A standard application of Jensen's formula reduces the proof to an estimate on the Hilbert--Schmidt norm of a certain operator whose Fredholm determinant is a multiple of the Selberg zeta function. The operator in question is given by
\[
\mathcal{L}_{\tau,\tau',\tau'';s}:=\mathcal{L}_{\tau,s}\mathcal{L}_{\tau',s}\mathcal{L}_{\tau'',s} \colon H^2(\Omega(h)) \to H^2(\Omega(h)),
\]
where the factors are refined transfer operators at the scales $\tau$, $\tau'$, and $\tau''$, and $H^2(\Omega(h))$ is the Bergman space over a small complex neighborhood $\Omega(h)$ of the limit set. The three factors play complementary roles. The outer factor $\mathcal{L}_{\tau,s}$ reduces the Hilbert--Schmidt norm to an integral over the limit set. The innermost factor $\mathcal{L}_{\tau'',s}$ is chosen at the resolution $\tau''\asymp h$, which separates the relevant branches into different connected components of the domain $\Omega(h)$. The middle factor $\mathcal L_{\tau',s}$ allows us to apply the Bourgain--Dyatlov Fourier decay theorem to exploit oscillatory cancellation as \(|t|=|\operatorname{Im}(s)|\) grows large. This separation of roles permits the parameters of the three refinements to be optimized independently.

\subsection{Notation}
We write $A \ll B$, $B\gg A$ or $A = O(B)$ when $|A| \le C|B|$ for some implied constant $C > 0$, and $A \asymp B$ when $A \ll B \ll A$. The notation $A \ll_y B$ means the implied constant may depend on the parameter $y$. We write $s = \sigma + it$ to mean that $\sigma = \mathrm{Re}(s)$ and $t = \mathrm{Im}(s)$. For a finite set $S$, we let $|S|$ or $\#S$ be its cardinality. Unless stated otherwise, all constants, implied or explicit, depend only on the Schottky data of the surface $X$, which we fix throughout; dependence on further parameters is indicated by subscripts.

\subsection{Use of AI}
AI language models were used to assist with checking arguments and improving exposition. The mathematical ideas are due to the author, except for the proof of the approximation result in Proposition~\ref{prop:approximation}. The author had established a weaker result and conjectured the stronger form stated here. At the author's request, GPT-6.1 Sol provided a proof using the Remez-type inequality in Lemma~\ref{lem:polynomial-norming}. The author critically reviewed and verified all suggestions and is solely responsible for the mathematical content.

\section{Preliminaries}\label{sec:Prelims}

\subsection{Schottky groups and combinatorial notation}
Let $\HH^2 = \{x + iy \in \CC : y > 0\}$ be the hyperbolic plane equipped with the hyperbolic metric $ds^2 = \frac{dx^2 + dy^2}{y^2}$. The group $\mathrm{SL}_2(\R)$ acts isometrically on $\HH^2$ via Möbius transformations:
$$
\gamma = \mat{a}{b}{c}{d} \in \mathrm{SL}_2(\R),\, z \in \HH^2   \Longrightarrow \gamma(z) = \frac{az + b}{cz + d}.
$$
By a theorem of Button \cite{Button} (see also \cite[Theorem~15.3]{Bor16}), every infinite-area convex cocompact hyperbolic surface arises as a quotient $X = \Gamma \backslash \HH^2$ for some Schottky group $\Gamma$. A Schottky group $\Gamma$ is constructed as follows. Fix some $m\in \mathbb{N}$ and an alphabet $\mathcal{A} = \{1, \dots, 2m\}$. For each $a \in \mathcal{A}$, define 
$$
\overline{a} = \begin{cases}a + m, & \text{if } a \le m,\\a - m, & \text{if } a > m.\end{cases}.
$$
Choose open disks $D_1, \dots, D_{2m} \subset \CC$ centered on $\mathbb{R}$ with pairwise disjoint closures. Fix isometries $\gamma_1, \dots, \gamma_{2m} \in \mathrm{SL}_2(\mathbb{R})$ such that $\gamma_a(\overline{\CC} \setminus D_{\overline{a}}) = \overline{D_a}$ and $\gamma_{\overline{a}} = \gamma_a^{-1}$ for all $a \in \mathcal{A}$. Then $\Gamma := \langle \gamma_1, \dots, \gamma_{2m} \rangle$ is a free group on $m$ generators.

\begin{figure}[ht]
\centering
\begin{tikzpicture}[scale=0.92]
  \draw (-0.6,0) -- (14.9,0);

  \draw (1.7,0) circle (1.45);
  \draw (5.0,0) circle (1.45);
  \draw (8.7,0) circle (1.88);
  \draw (12.8,0) circle (1.45);

  \draw (0.8,0) circle (0.28);
  \draw (1.7,0) circle (0.28);
  \draw (2.6,0) circle (0.28);

  \draw (4.1,0) circle (0.28);
  \draw (5.0,0) circle (0.28);
  \draw (5.9,0) circle (0.28);

  \draw (7.5,0) circle (0.33);
  \draw (8.7,0) circle (0.33);
  \draw (9.9,0) circle (0.33);

  \draw (11.9,0) circle (0.28);
  \draw (12.8,0) circle (0.28);
  \draw (13.7,0) circle (0.28);

  \node at (0.95,1.42) {$D_1$};
  \node at (5.75,1.42) {$D_3$};
  \node at (8.0,1.72) {$D_2$};
  \node at (13.45,1.42) {$D_4$};

  \node[below] at (0.8,-0.3) {$D_{12}$};
  \node[below] at (1.7,-0.3) {$D_{14}$};
  \node[below] at (2.6,-0.3) {$D_{11}$};

  \node[below] at (4.1,-0.3) {$D_{33}$};
  \node[below] at (5.0,-0.3) {$D_{32}$};
  \node[below] at (5.9,-0.3) {$D_{34}$};

  \node[below] at (7.5,-0.35) {$D_{21}$};
  \node[below] at (8.7,-0.35) {$D_{23}$};
  \node[below] at (9.9,-0.35) {$D_{22}$};

  \node[below] at (11.9,-0.3) {$D_{44}$};
  \node[below] at (12.8,-0.3) {$D_{41}$};
  \node[below] at (13.7,-0.3) {$D_{43}$};

  \draw[<-] (2.0,2.05) .. controls (2.9,2.75) and (4.0,2.75) .. (5.0,2.05);
  \draw[<-] (5.0,2.4) .. controls (4.1,3.0) and (2.9,3.0) .. (2.2,2.4);
  \node at (2.9,1.95) {$\gamma_1$};
  \node at (3.4,3.08) {$\gamma_3$};

  \draw[<-] (9.35,2.0) .. controls (10.45,2.65) and (11.75,2.65) .. (13.05,2.02);
  \draw[<-] (13.0,2.4) .. controls (11.9,3.05) and (10.55,3.05) .. (9.45,2.38);
  \node at (10.9,1.92) {$\gamma_2$};
  \node at (11.35,3.12) {$\gamma_4$};
\end{tikzpicture}
\caption{A Schottky configuration with $m=2$. The disks $D_1, D_3$ and $D_2, D_4$ are paired by $\gamma_1$ and $\gamma_2$, respectively. The inner disks $D_{ab} := \gamma_a(D_b)$ with $a,b \in \mathcal{A}= \{1,2,3,4\}$ and $a \to b$ are also shown.}
\label{fig:schottky_disks}
\end{figure}
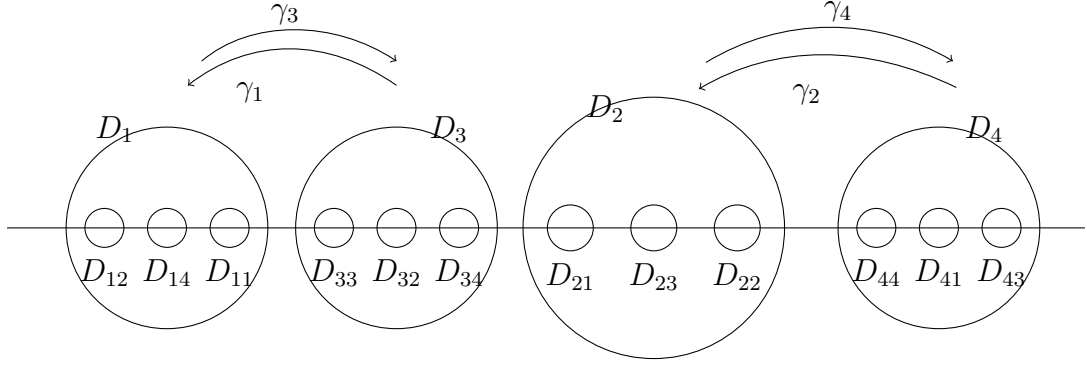

We will use the following combinatorial notation from \cite{DZ20}. A word $\Aa = a_1 \cdots a_n$ is a finite sequence of letters $a_j \in \mathcal{A}$, called reduced if $a_j \ne \overline{a_{j+1}}$ for all $1 \le j \le n-1$. For any word $\Aa = a_1 \cdots a_n$, we define $\gamma_\Aa := \gamma_{a_1} \cdots \gamma_{a_n}$, with the empty word $\emptyset$ corresponding to the identity element $\gamma_\emptyset := \mathrm{id}$. For any word $\Aa$, the reduced word $\mathsf{red}(\Aa) \in \mathcal{W}$ is obtained by cancelling all trivial relations $a\overline{a} = \emptyset$, and we have $\gamma_\Aa = \gamma_{\mathsf{red}(\Aa)}.$

Let $\mathcal{W}_n$ denote the set of reduced words of length $n$, with $\mathcal{W}_0 = \{\emptyset\}$. Let $\mathcal{W} := \bigcup_{n \ge 0} \mathcal{W}_n$ be the set of all reduced words (including the empty word). Thanks to the freeness of $\Gamma$, we have a bijection between $\mathcal{W} $ and $\Gamma$ given by $\Aa \mapsto \gamma_{\Aa}$. Let $\mathcal{W}_{\le n} := \bigcup_{k \le n} \mathcal{W}_k$ and similarly for $\mathcal{W}_{\ge n}$. Let $\mathcal{W}^\circ := \mathcal{W} \setminus \{\emptyset\}$ denote the set of non-empty reduced words. 

For $\Aa = a_1 \cdots a_n \in \mathcal{W}^\circ$, denote $\Aa' := a_1 \cdots a_{n-1}$. Let $\Bb = b_1 \cdots b_m \in \mathcal{W}^\circ$. We write $\Aa \to \Bb$ if and only if $a_n \ne \overline{b_1}$, which is equivalent to saying that the concatenation $\Aa\Bb$ is reduced. We write $\Aa \prec \Bb$ if $\Aa$ is a prefix of $\Bb$ (including equality). We write $\overline{\Aa} := \overline{a_n} \cdots \overline{a_1}$, so $\gamma_{\overline{\Aa}} = \gamma_\Aa^{-1}$.

For $b \in \mathcal{A}$, we write $\Aa \rightsquigarrow b$ if the last letter of $\Aa$ is $b$. We say that $\Aa, \Bb \in \mathcal{W}$ are prefix-comparable, respectively suffix-comparable, if
\[
\Aa \sim_{\mathrm{pre}} \Bb \Longleftrightarrow \Aa \prec \Bb \text{ or } \Bb \prec \Aa,
\qquad
\Aa \sim_{\mathrm{suf}} \Bb \Longleftrightarrow \overline{\Aa} \sim_{\mathrm{pre}} \overline{\Bb}.
\]
We write $\Aa \nsim_{\mathrm{pre}} \Bb$ and $\Aa \nsim_{\mathrm{suf}} \Bb$ for the negations of these relations.

For each $\Aa \in \mathcal{W}^\circ$, define the disk $D_\Aa := \gamma_{\Aa'}(D_{a_n})$ and the interval $I_\Aa := D_\Aa \cap \mathbb{R}$. We have: $\Aa \prec \Bb \Rightarrow D_\Bb \subset D_\Aa$ and $\Aa \nsim_{\mathrm{pre}} \Bb \Rightarrow D_\Bb \cap D_\Aa = \emptyset$; moreover, the size of the intervals decays exponentially with the word length: there exist constants $C, \theta > 0$ with $\theta < 1$ such that
\[
  |I_\Aa| \le C\theta^{|\Aa|}.
\]
Recall that the limit set is defined as the set of accumulation points of the orbit $\Gamma z$ for any $z \in \HH^2$. In terms of the combinatorial notation, the limit set can be re-expressed as $\Lambda := \bigcap_{n \ge 1} \bigsqcup_{\Aa \in \mathcal{W}_n} I_\Aa \subset \mathbb{R}$. Let $\delta \in (0,1)$ be its Hausdorff dimension.

\textit{Throughout the remainder of the paper, fix a non-elementary Schottky group $\Gamma$, its Schottky data $\{D_a,\gamma_a\}_{a\in\mathcal A}$, and the surface $X=\Gamma\backslash\HH^2$. Thus $m\ge2$ and $\delta\in(0,1)$.}

\subsection{The standard transfer operator}\label{sec:classicalTransfOper}
For a non-empty open set $\Omega \subset \CC$, the \textit{Bergman space} over $\Omega$ is given by
\begin{equation}\label{defi_bergman}
H^2(\Omega) := \left\{ f \colon \Omega \to \CC \text{ holomorphic} : \Vert f\Vert_{L^2(\Omega)} := \left( \int_\Omega |f(z)|^2 \, \dvol(z) \right)^{1/2} < \infty \right\}.
\end{equation}
Here $\dvol$ denotes standard two-dimensional Lebesgue measure.

Let $D = \bigsqcup_{a \in \mathcal{A}} D_a$ be the union of all Schottky disks. The (standard) \textit{transfer operator} $\mathcal{L}_s \colon H^2(D) \to H^2(D)$ is defined by
\begin{equation}\label{transf_defi}
\mathcal{L}_s f(z) := \sum_{\substack{a \in \mathcal{A} \\ a \to b}} \gamma_a'(z)^s f(\gamma_a(z)), \quad z \in D_b.
\end{equation}
Recall that $a \to b$ means $a \ne \overline{b}$. For such $a,b$, the pole of $\gamma_a$ is real and lies outside $\overline{D_b}$, so $\gamma_a'(D_b)$ avoids $(-\infty,0]$ and $\gamma_a'(x)>0$ for $x\in I_b$. We therefore define $\gamma_a'(z)^s := \exp(s\mathbb L(\gamma_a'(z)))$, where $\mathbb L(z):=\log|z|+i\arg(z)$ is the principal branch of the complex logarithm.

The transfer operator $\mathcal{L}_s$ is a holomorphic family of trace class operators, and its Fredholm determinant is equal to the Selberg zeta function of $X$,
\begin{equation}\label{eq:fredholmSelberg}
\det(1 - \mathcal{L}_s) = Z_X(s),
\end{equation}
see \cite[Chapter~15]{Bor16} for details. The Selberg zeta function is classically defined for $\mathrm{Re}(s) > \delta$ by the infinite product
\[
Z_X(s) := \prod_{k=0}^\infty \prod_{\ell \in \mathcal{L}_X} \left(1 - e^{-(s+k)\ell}\right),
\]
where $\mathcal{L}_X$ is the multiset of lengths of primitive closed geodesics on $X$, including length multiplicities. By a result of Patterson--Perry \cite{PaPe01} (see also \cite{BJP03}), the resonances coincide with the zeros of the Selberg zeta function $Z_X(s)$ with the same multiplicity, except possibly in $-\mathbb{N}_0$, where topological zeros occur. Thus, counting resonances (outside $-\mathbb{N}_0$) is equivalent to counting zeros of $Z_X(s)$. From \eqref{eq:fredholmSelberg}, it follows that a zero of $Z_X(s)$ corresponds to the existence of a non-trivial solution $f\in H^2(D)$ to the equation $\mathcal{L}_s f = f$, i.e., a 1-eigenfunction of $\mathcal{L}_s$.

\subsection{Thermodynamic formalism}\label{sec:thermoFormalism}
We now make the connection between the transfer operator formalism and thermodynamic formalism. Let $I = \bigsqcup_{a \in \mathcal{A}} I_a$ be the union of intervals corresponding to the Schottky disks and consider the map $T \colon I \to \overline{\mathbb{R}}$ defined by $T(x) = \gamma_a^{-1}(x)$ for $x \in I_a$. The map $T$ is Markov and uniformly expanding after sufficiently many iterates on its non-escaping domain, with inverse branches given by $\gamma_a$ for $a \in \mathcal{A}$. The limit set is its non-escaping set, $\Lambda = \bigcap_{n \ge 0} T^{-n}(I)$. This map encodes the dynamics of the group $\Gamma$; the origins of this type of coding go back to Bowen--Series~\cite{BowenSeries1}. The transfer operator $\mathcal{L}_s$ defined in \eqref{transf_defi} is a complexified version of the Ruelle transfer operator associated to $T$ with potential $\varphi  = -s \log |T'|$, where $s = \sigma + it$:
\[
\mathcal{L}_s f(x) := \sum_{T y = x} |T'(y)|^{-s} f(y), \quad x \in I.
\]
For a continuous function $\varphi \colon I \to \mathbb{R}$, the topological pressure is defined by
\begin{equation}\label{pressure_def}
P(\varphi) := \lim_{N \to \infty} \frac{1}{N} \log \sum_{T^N x = x} e^{\sum_{k=0}^{N-1} \varphi(T^k x)}.
\end{equation}
By a result of Bowen \cite{Bowen1}, $P(\sigma) := P(-\sigma \log |T'|)$ is convex and strictly decreasing on $\mathbb{R}$, with $P(\delta) = 0$ where $\delta$ is the Hausdorff dimension of $\Lambda$. The connection to the transfer operator defined above is given by the following result:

\begin{prop}[Ruelle--Perron--Frobenius]\label{prop:RPF}
For $\sigma \in \mathbb{R}$, the transfer operator $\mathcal{L}_\sigma$ on $C^1(\overline{I})$ is a positive operator with spectral radius $e^{P(\sigma)}$, which is a simple eigenvalue with strictly positive eigenfunction $\varphi_\sigma$. The essential spectral radius is strictly smaller than $e^{P(\sigma)}$. The spectral projector onto $\{e^{P(\sigma)}\}$ is given by
\[
\mathbb{P}_\sigma(f) = \varphi_\sigma \int_{\Lambda} f \, d\mu_\sigma,
\]
where $\mu_\sigma$ is the unique probability measure satisfying $\mathcal{L}_\sigma^* \mu_\sigma = e^{P(\sigma)} \mu_\sigma$. Moreover, the measure $\mu_\sigma$ is supported on $\Lambda$, and we normalize $\varphi_\sigma$ by $\int_\Lambda\varphi_\sigma\,d\mu_\sigma=1$.
\end{prop}

\noindent The formulation of this result in the symbolic setting appears in \cite[Theorem~2.2]{ParryPollicott1990}. The version stated here is the same as in \cite[Theorem~5.1]{Naud2}, and can be derived from \cite{Bowen1975,Liverani1995}.

At $\sigma=\delta$, the measure $\mu_\delta$ is Ahlfors $\delta$-regular:
\[
\mu_\delta((x-r,x+r))\asymp r^\delta,\qquad x\in\Lambda,\quad 0<r\le1.
\]
This follows from \cite[Lemma~14.13]{Bor16} and the relation to the Patterson--Sullivan measure in \cite[Lemma~4.7]{JNS20}.

\subsection{Partitions and refined transfer operators}\label{sec:partRefOp}
For a set $Z \subset \mathcal{W}$, define its mirror set by $\overline{Z} := \{ \overline{\Aa} : \Aa \in Z \}$. To each finite set $Z \subset \mathcal{W}_{\ge 2}$, we associate an operator $\mathcal{L}_{Z,s}$ on $H^2(D)$ defined by
\begin{equation}\label{generalizedOp}
\mathcal{L}_{Z,s}f(z) := \sum_{\substack{\Aa \in \overline{Z} \\ \Aa \rightsquigarrow b}} \gamma_{\Aa'}'(z)^s f(\gamma_{\Aa'}(z)), \quad z \in D_b.
\end{equation}

For $b \in \mathcal{A}$, set
\[
Y_b(Z) := \{ \Aa \in \mathcal{W} : \Aa b \in \overline{Z} \}.
\]
Then \eqref{generalizedOp} can be re-expressed as
\[
\mathcal{L}_{Z,s}f(z) = \sum_{\Aa \in Y_b(Z)} \gamma_{\Aa}'(z)^s f(\gamma_{\Aa}(z)), \quad z \in D_b,
\]
and we will use this notation for the rest of the paper.

We now recall the definition of partitions from the work of Dyatlov--Zworski \cite{DZ20}. A finite set $Z \subset \mathcal{W}^\circ$ is a \textit{partition} if there exists $N\in\mathbb{N}$ such that for all $\Aa\in \mathcal{W}$ with word length $\vert \Aa \vert \ge N$, there exists a unique $\Bb \in Z$ such that $\Bb \prec \Aa$. In terms of the limit set, $Z$ is a partition if we have the \textit{disjoint} union
\[
\Lambda = \bigsqcup_{\Bb \in Z} (I_{\Bb} \cap \Lambda).
\]
The set $\mathcal{W}_n$ is a partition for each $n \ge 2$. In this case, the associated transfer operator is the $(n-1)$-th iterate of the standard transfer operator, that is, $\mathcal{L}_{\mathcal{W}_n,s} = \mathcal{L}_s^{n-1}.$

We are interested in the following family of partitions, introduced in \cite{DZ20}. For $0 < \tau < \min_{a \in \mathcal{A}} |I_a|$, define
\begin{equation}\label{defi:Z}
Z(\tau):= \{ \Aa\in \mathcal{W}_{\ge 2} : \vert I_{\Aa}\vert \le  \tau < \vert I_{\Aa'}\vert \}.
\end{equation}
We will call the associated transfer operator $\mathcal{L}_{\tau,s} := \mathcal{L}_{Z(\tau),s}$ the $\tau$-refined transfer operator. Set
\begin{equation}\label{defi:Y}
Y_b(\tau) := Y_b(Z(\tau)), \qquad Y(\tau) := \bigcup_{b \in \mathcal{A}} Y_b(\tau) = (\overline{Z(\tau)})'.
\end{equation}
Then \eqref{generalizedOp} reads
\[
\mathcal{L}_{\tau,s}f(z) := \sum_{\Aa \in Y_b(\tau)} \gamma_{\Aa}'(z)^s f(\gamma_{\Aa}(z)), \quad z \in D_b.
\]
The key property of partitions, established in \cite{DZ20}, is that any 1-eigenfunction of $\mathcal{L}_s$ is also a 1-eigenfunction of $\mathcal{L}_{Z,s}$ for any partition $Z$. We will need a slightly stronger version of this fact, which is the content of the following lemma.

\begin{lemma}\label{lem:fredholmRefined}
For any partition $Z \subset \mathcal{W}_{\ge 2}$, there exists a holomorphic family of trace class operators $\mathcal{T}_{Z,s}$ such that $1-\mathcal{L}_{Z,s} = (1-\mathcal{T}_{Z,s})(1-\mathcal{L}_{s}),$ where $\mathcal{L}_s$ is the standard transfer operator. More generally, for any finite collection of partitions $Z_1, \dots, Z_k \subset \mathcal{W}_{\ge 2}$, there exists a holomorphic family of trace class operators $\mathcal{T}_{s}$, depending on $Z_1, \dots, Z_k$, such that
\[
1 - \mathcal{L}_{Z_1,s} \cdots \mathcal{L}_{Z_k,s} = (1 - \mathcal{T}_{s})(1 - \mathcal{L}_s).
\] 
In particular, the Fredholm determinant of $\mathcal{L}_{Z_1,s} \cdots \mathcal{L}_{Z_k,s}$ is a holomorphic multiple of the Selberg zeta function $Z_X(s)$, i.e., there exists an entire function $H(s)$, depending on $Z_1, \dots, Z_k$, such that
\[
\det(1 - \mathcal{L}_{Z_1,s} \cdots \mathcal{L}_{Z_k,s}) = Z_X(s) H(s).
\]
\end{lemma}

The statement is a combination of Lemmas 2.3 and 2.6 from \cite{Soa25}, and the proof given there is a straightforward adaptation of the arguments in \cite{DZ20}.

\subsection{Bounds for Schottky groups}\label{sec:boundsForDeriv}
We now state some bounds for Schottky groups that will be used throughout. Let us introduce some notation first, following \cite{NM20}. For each $a \in \mathcal{A}$, fix a point $o_a \in I_a$. If $\textbf{a} = a_1\cdots a_n \in \mathcal{W}^\circ$, set $o_{\textbf{a}} := o_{a_n}$, and let $o_\emptyset \in D$ be arbitrary. Thus, the reference point $o_{\textbf{a}}$ is uniquely determined by the terminal letter of $\textbf{a}$. These reference points will be fixed throughout the paper. We now define
\begin{equation}\label{defi:Upsilon}
\Upsilon_{\textbf{a}} := \gamma_{\textbf{a}}'(o_{\textbf{a}}).
\end{equation}
In particular, $\Upsilon_\emptyset=1$.

We begin with the following standard fact:

\begin{lemma}[Bounded distortion]\label{lem:bounded_distortion}
For all $\Aa = a_1\cdots a_n \in \mathcal{W}^\circ$ and $z \in D_{b}$ with $\Aa \to b$, we have
\[
\frac{\gamma''_{\Aa}(z)}{\gamma'_{\Aa}(z)} \asymp 1.
\]
\end{lemma}

\begin{proof}
We can give a direct proof by writing $\gamma_{\Aa}$ as a matrix in $\mathrm{SL}_2(\mathbb{R})$:
\[
\gamma_{\Aa} = \pmat{a_{\Aa}}{b_{\Aa}}{c_{\Aa}}{d_{\Aa}}.
\]
We can then compute
\[
\gamma_{\Aa}'(z) = \frac{1}{(c_{\Aa} z + d_{\Aa})^2}
= \frac{1}{c_{\Aa}^2 (z - x_{\Aa})^2}, \quad \text{where } x_{\Aa} = -\frac{d_{\Aa}}{c_{\Aa}} = \gamma_{\Aa}^{-1}(\infty),
\]
and
\begin{equation}\label{eq:log_der}
\frac{\gamma_{\Aa}''(z)}{\gamma_{\Aa}'(z)} = -\frac{2}{z - x_{\Aa}}.
\end{equation}
Since $\infty$ is not contained in any of the Schottky disks $D_b$, the point $x_{\Aa} = \gamma_{\Aa}^{-1}(\infty)$ lies inside $D_{\overline{a}_n}$. Recall that $\Aa \to b$ means that $\overline{a}_n \neq b$. Hence, since the Schottky disks have mutually disjoint closures, we have $|z - x_{\Aa}| \asymp 1$ for $z \in D_b$ with $\Aa \to b$, which implies the desired bounded distortion property.
\end{proof}

We use the following derivative bounds; see \cite[Lemmas~3.1--3.5]{NM20}.

\begin{lemma}[Derivative bounds]\label{lem:boundsderivatives}
We have the following:
\begin{enumerate}[(i)]
\item For $b \in \mathcal{A}$, $z \in D_b$, $\Aa \to b$: $|\gamma'_{\Aa}(z)| \asymp |I_\Aa| \asymp \Upsilon_{\Aa}$.
\item $\Upsilon_{\Aa'} \asymp \Upsilon_{\Aa}$ for all $\Aa \in \mathcal{W}^\circ$.
\item $\Upsilon_{\overline{\Aa}} \asymp \Upsilon_{\Aa}$ for all $\Aa \in \mathcal{W}^\circ$.
\item $\Upsilon_{\Aa\Bb} \asymp \Upsilon_{\Aa} \Upsilon_{\Bb}$ for $\Aa, \Bb \in \mathcal{W}^\circ$ with $\Aa \to \Bb$.
\end{enumerate}
\end{lemma}

The following bounds follow from \cite[Lemmas~2.10 and~2.13]{BD17} for $Z(\tau)$, and then for $Y(\tau)$ from Lemma~\ref{lem:boundsderivatives}(ii)--(iii); part (iii) follows from (i) and Lemma~\ref{lem:boundsderivatives}(iv).

\begin{lemma}[Partition estimates]\label{lem:YZ}
For all sufficiently small $\tau > 0$ we have the following bounds:
\begin{enumerate}[(i)]
\item For all $\Aa \in Z(\tau) \cup Y(\tau)$ we have $\Upsilon_{\Aa} \asymp \tau$.
\item $|Z(\tau)| \asymp |Y(\tau)| \asymp \tau^{-\delta}$.
\item For all $\Aa \in Z(\tau) \cup Y(\tau)$ we have $|\Aa| \asymp \log(\tau^{-1})$, where $|\Aa|$ is the reduced length of the word $\Aa$.
\end{enumerate}
\end{lemma}

\subsection{Refined function spaces}\label{sec:refinedFuncti}
We work with refined function spaces following Guillop\'{e}--Lin--Zworski \cite{GLZ}. Fix a small parameter $h > 0$ and let $\Lambda(h) := \Lambda + (-h, h)$ be the real $h$-neighborhood of the limit set $\Lambda$, i.e., the set of points in $\mathbb{R}$ that are within distance $h$ of $\Lambda$. Let $N(h)$ denote the number of connected components of $\Lambda(h)$, and denote them by $I_\ell(h)$, $\ell = 1, \dots, N(h)$. From \cite[Lemma~15.14]{Bor16} and \cite[Equation~15.25]{Bor16} we have the following structural result for $\Lambda(h)$:

\begin{lemma}[Structure of $\Lambda(h)$]\label{lem:structureLambdah}
For sufficiently small $h > 0$, we have $N(h) \le Ch^{-\delta}$ and each interval $I_\ell(h)$ has diameter at most $Ch$, for some constant $C = C(\Gamma)> 0$.
\end{lemma}

Let $I_\ell(h)$ be the intervals from Lemma \ref{lem:structureLambdah} and let $D_\ell(h)$ be the open Euclidean disk centered on $\mathbb{R}$ such that $D_\ell(h) \cap \mathbb{R} = I_\ell(h)$. We define the refined domain $\Omega(h)$ as
\begin{equation}\label{defi:omega(h)}
\Omega(h) := \bigsqcup_{\ell=1}^{N(h)} D_\ell(h).
\end{equation}
Since $\Lambda$ is compact, we have $\Omega(h) \subset D$ for all sufficiently small $h > 0$. We will work with the refined function space $H^2(\Omega(h))$ defined as in \eqref{defi_bergman}. For $b \in \mathcal{A}$, set $\Omega_b(h) := \Omega(h) \cap D_b$. We define the $\tau$-refined transfer operator on $H^2(\Omega(h))$ by
\begin{equation}\label{refOpRefSpace}
\mathcal{L}_{\tau,s}f(z) := \sum_{\Aa \in Y_b(\tau)} \gamma'_{\Aa}(z)^s f(\gamma_{\Aa}(z)), \quad z \in \Omega_b(h).
\end{equation}

We include the short argument, which is due to Naud~\cite[Lemma~3.2]{Naud14}.

\begin{lemma}[Contraction in $\Omega(h)$]\label{lem:contractionInH}
There exist $N_0 \in \mathbb{N}$ and $h_0 > 0$, depending only on $\Gamma$, such that for all $0 < h < h_0$ and $\Aa \in \mathcal{W}_{\ge N_0}$ with $\Aa \to b$,
\[
\gamma_{\Aa}(\Omega_b(h)) \subset \Omega(h/2).
\]
\end{lemma}

\begin{proof}
Let $D_\ell(h)\subset D_b$ be a component of $\Omega_b(h)$ and pick $x\in\Lambda\cap I_\ell(h)$. Since $\gamma_\Aa$ is a real Möbius map whose pole lies outside $\overline{D_b}$, the image $\gamma_\Aa(D_\ell(h))$ is the open disk centered on $\mathbb{R}$ with real diameter $\gamma_\Aa(I_\ell(h))$. This interval contains $\gamma_\Aa(x)\in\Lambda$ and, by Lemmas~\ref{lem:boundsderivatives}(i) and~\ref{lem:structureLambdah}, has length $\ll\Upsilon_\Aa h\ll\theta^{|\Aa|}h$. For $N_0$ large enough, it is therefore contained in $(\gamma_\Aa(x)-h/2,\gamma_\Aa(x)+h/2)$, hence in a single component $I_{\ell'}(h/2)$ of $\Lambda(h/2)$. Since a disk centered on $\mathbb{R}$ is contained in any other such disk whose real diameter contains its own, we get $\gamma_\Aa(D_\ell(h))\subset D_{\ell'}(h/2)\subset\Omega(h/2)$.
\end{proof}

Note that the same conclusion holds for all $\Aa \in Y(\tau)$ with $\Aa \to b$ and $\tau > 0$ small enough, since $Y(\tau) \subset \mathcal{W}_{\ge N_0}$ for all $\tau > 0$ small enough by Lemma~\ref{lem:YZ}(iii). In other words, the maps $\gamma_\Aa$ with $\Aa \in Y(\tau)$ are contracting on $\Omega(h)$ for all $\tau$ and $h$ sufficiently small. Consequently, the $\tau$-refined transfer operator \eqref{refOpRefSpace} is well-defined as an operator from $H^2(\Omega(h))$ to $H^2(\Omega(h))$, and it is trace class, for all $\tau$ and $h$ sufficiently small.

\subsection{Bounds for $\Omega(h)$}\label{sec:boundsForOmega}
We now record some bounds for the refined domain $\Omega(h)$. First we need the following volume bound, which follows directly from Lemma~\ref{lem:structureLambdah} and the definition of $\Omega(h)$:

\begin{lemma}\label{lem:vol_bound_omega}
For all $h > 0$ small enough, $\vol(\Omega(h)) \le Ch^{2-\delta}$.
\end{lemma}

Next, we control the size of the complex powers $\gamma_{\Aa}'(z)^s$ for $z \in \Omega(h)$.

\begin{lemma}\label{lem:derivative_bounds_omega} 
For all $h > 0$ small enough, for all $b \in \mathcal{A}$, $s = \sigma + it$, $z = x+iy \in \Omega_b(h)$, and $\Aa \in \mathcal{W}^\circ$ with $\Aa \to b$, we have the following estimates:
\begin{enumerate}[(i)]
\item $\mathbb{L}(\gamma_{\Aa}'(z)) = \log(\gamma_{\Aa}'(x)) + O(h),$
\item $|\gamma_{\Aa}'(z)^s| \ll e^{C(h|t| + |\sigma|)}  \Upsilon_{\Aa}^\sigma.$
\item If we assume in addition that $\Aa \in Y(\tau)\cup Z(\tau)$, then we have $ |\gamma_{\Aa}'(z)^s| \ll  e^{C(h|t| + |\sigma|)} \tau^\sigma.$
\end{enumerate}
\end{lemma}

\begin{remark}
We will eventually choose $h$ to be roughly $|t|^{-1}$, so the exponential factor $e^{C(h|t| + |\sigma|)}$ will be bounded by a constant multiple of $e^{C|\sigma|}$, eliminating exponential growth in $t$. This is the reason for working with refined spaces.
\end{remark}

\begin{proof}[Proof of Lemma~\ref{lem:derivative_bounds_omega}]
The bound in (iii) follows from (ii) and Lemma~\ref{lem:YZ}(i). To prove (ii), we write
\[
\gamma_{\Aa}'(z)^s = \gamma_{\Aa}'(z)^{\sigma} \exp ( it \mathbb{L}(\gamma_{\Aa}'(z)) ).
\]
Taking absolute values and applying the bound from (i) gives
\[
|\gamma_{\Aa}'(z)^s| = |\gamma_{\Aa}'(z)|^{\sigma} \exp ( -t \,\mathrm{Im}(\mathbb{L}(\gamma_{\Aa}'(z))) ) = |\gamma_{\Aa}'(z)|^{\sigma} \exp ( O(h|t|) ).
\]
Applying Lemma~\ref{lem:boundsderivatives}(i) gives (ii). Finally, the bound in (i) follows by integrating the logarithmic derivative along the segment from $x$ to $z$ and using bounded distortion:
\begin{align*}
|\mathbb{L}(\gamma_{\Aa}'(z)) - \log(\gamma_{\Aa}'(x))| &= |\mathbb{L}(\gamma_{\Aa}'(z)) - \mathbb{L}(\gamma_{\Aa}'(x))| \\
&\leq \sup_{w \in \Omega_b(h)} \left|\frac{\gamma_{\Aa}''(w)}{\gamma_{\Aa}'(w)}\right|   |z-x| \\
&= O(h).
\end{align*}
This completes the proof.
\end{proof}

\subsection{Separation lemma}\label{sec:separation}
The following separation lemma is a key ingredient in the proof of the main results of this paper. Recall that $\Aa \sim_{\mathrm{pre}} \Bb$ means that $\Aa$ and $\Bb$ are prefix-comparable, i.e. one of them is a prefix of the other.

\begin{lemma}[Separation]\label{lem:separation}
There exist $C_{\mathrm{sep}}, h_0 > 0$ depending only on $\Gamma$ such that for $0 < h < h_0$ and $\tau \ge C_{\mathrm{sep}} h$ the following holds true: for all $\ell_1, \ell_2 \in \mathcal{A}$, $\Aa \in Y_{\ell_1}(\tau)$, $\Bb \in Y_{\ell_2}(\tau)$, $z_1 \in D_{\ell_1}$, and $z_2 \in D_{\ell_2}$,
\[
\gamma_{\Aa}(z_1),\, \gamma_{\Bb}(z_2) \text{ lie in the same connected component of } \Omega(h) \quad \Longrightarrow \quad \Aa \sim_{\mathrm{pre}} \Bb.
\]
\end{lemma}

A version of this lemma appears as \cite[Lemma~3.1]{Soa25} and is an adaptation of \cite[Lemma~4.4]{JN2016}. Let us give a proof for completeness.

\begin{proof}
Fix words $\Aa \in Y_{\ell_1}(\tau)$ and $\Bb \in Y_{\ell_2}(\tau)$. Suppose for contradiction that $\Aa$ is not a prefix of $\Bb$ and $\Bb$ is not a prefix of $\Aa$, yet $\gamma_{\Aa}(z_1)$ and $\gamma_{\Bb}(z_2)$ lie in the same connected component of $\Omega(h)$ for some $z_1 \in D_{\ell_1}$, $z_2 \in D_{\ell_2}$. By Lemma~\ref{lem:structureLambdah}, all points in a single connected component of $\Omega(h)$ lie within distance $Ch$ of each other, so
\begin{equation}\label{eq:separation1}
|\gamma_{\Aa}(z_1) - \gamma_{\Bb}(z_2)| \le Ch.
\end{equation}
Let $\Cc$ be the maximal common prefix of $\Aa$ and $\Bb$. If $\Cc = \emptyset$, then $\Aa$ and $\Bb$ have distinct first letters, and so $\gamma_{\Aa}(z_1)$ and $\gamma_{\Bb}(z_2)$ lie in distinct Schottky disks. Since the Schottky disks have mutually disjoint closures, this contradicts \eqref{eq:separation1} for sufficiently small $h$. Thus, we may assume that the prefix $\Cc$ is a non-empty word. Since $\Aa$ and $\Bb$ are not prefixes of each other, we can write $\Aa = \Cc\Aa_2$ and $\Bb = \Cc\Bb_2$, where $\Aa_2$ and $\Bb_2$ are non-empty words with different first letters. Then $\gamma_{\Aa_2}(z_1)$ and $\gamma_{\Bb_2}(z_2)$ lie in distinct Schottky disks, so
\[
|\gamma_{\Aa_2}(z_1) - \gamma_{\Bb_2}(z_2)| \gg 1.
\]
Using the identity $|\gamma(u) - \gamma(w)| = |\gamma'(u)|^{1/2} |\gamma'(w)|^{1/2} |u - w|$, valid for all M\"obius transformations $\gamma$ and points $u,w\in \mathbb{C}$, and using also Lemma~\ref{lem:boundsderivatives}(i), we get
\begin{align*}
|\gamma_{\Aa}(z_1) - \gamma_{\Bb}(z_2)| & = |\gamma_{\Cc}(\gamma_{\Aa_2}(z_1)) - \gamma_{\Cc}(\gamma_{\Bb_2}(z_2))| \\
& \asymp \Upsilon_\Cc \cdot |\gamma_{\Aa_2}(z_1) - \gamma_{\Bb_2}(z_2)| \\
& \gg \Upsilon_\Cc.
\end{align*}
Since $\Cc \prec \Aa$, Lemma~\ref{lem:boundsderivatives}(iv) and Lemma~\ref{lem:YZ}(i) give $\Upsilon_\Cc \gg \Upsilon_\Aa \asymp \tau$. Thus we have shown $|\gamma_{\Aa}(z_1) - \gamma_{\Bb}(z_2)| \gg \tau$, which contradicts \eqref{eq:separation1} if $\tau \ge C_{\mathrm{sep}} h$ for some constant $C_{\mathrm{sep}} > 0$ large enough.
\end{proof}

Pairs $(\Aa, \Bb)$ of prefix- and suffix-comparable words in $Y(\tau)\times Y(\tau)$ can be controlled using the next lemma.

\begin{lemma}[Bounded transitions]\label{lem:boundedTransitions}
There exists a finite set $F \subset \mathcal{W}^\circ$, depending only on $\Gamma$, such that for every sufficiently small $\tau > 0$ and all distinct words $\Aa, \Bb \in Y(\tau)$:
\begin{enumerate}[(i)]
\item if $\Aa \sim_{\mathrm{pre}} \Bb$, then either $\Bb = \Aa\mathbf{v}$ or $\Aa = \Bb\mathbf{v}$ for some $\mathbf{v} \in F$;
\item if $\Aa \sim_{\mathrm{suf}} \Bb$, then either $\Bb = \mathbf{v}\Aa$ or $\Aa = \mathbf{v}\Bb$ for some $\mathbf{v} \in F$.
\end{enumerate}
\end{lemma}

\begin{proof}
For (i), after interchanging $\Aa$ and $\Bb$ if necessary, write $\Bb = \Aa\mathbf{v}$ with $\mathbf{v} \in \mathcal{W}^\circ$, where the concatenation is reduced. Lemmas~\ref{lem:boundsderivatives}(iv) and~\ref{lem:YZ}(i) give
\[
\tau \asymp \Upsilon_\Bb \asymp \Upsilon_\Aa \Upsilon_{\mathbf{v}} \asymp \tau \Upsilon_{\mathbf{v}}.
\]
Hence $\Upsilon_{\mathbf{v}} \asymp 1$. Since $\Upsilon_{\mathbf{v}} \asymp |I_{\mathbf{v}}| \ll \theta^{|\mathbf{v}|}$, the length of $\mathbf{v}$ is bounded in terms of $\Gamma$. Taking $F$ to contain every non-empty reduced word up to this length proves (i). Part (ii) follows by applying the same argument to $\Bb = \mathbf{v}\Aa$.
\end{proof}

\subsection{Phase function and its derivative}\label{sec:phaseder}
In the proof of Theorem~\ref{thm:main}, we will naturally encounter expressions of the form $\gamma_{\Aa}'(z)^s \overline{\gamma_{\Bb}'(z)^s}$ for $s = \sigma + it$ and $\Aa, \Bb \in \mathcal{W}^\circ$. We will need to understand the oscillatory behavior of such expressions as $t$ grows large. To this end, we introduce the phase function
\begin{equation}\label{phase}
\Phi_{\Aa, \Bb}(z) := \overline{\mathbb{L}(\gamma_{\Bb}'(z))} - \mathbb{L}(\gamma_{\Aa}'(z)),
\end{equation}
where $\mathbb{L}$ is the principal branch of the complex logarithm. This gives
\[
\gamma_{\Aa}'(z)^s \overline{\gamma_{\Bb}'(z)^s} = \gamma_{\Aa}'(z)^{\sigma} \overline{\gamma_{\Bb}'(z)}^{\sigma} e^{-it\Phi_{\Aa,\Bb}(z)}.
\]
For $\Aa, \Bb \in \mathcal{W}^\circ$, define the quantity
\begin{equation}\label{Dab}
\mathcal{D}_{\Aa, \Bb} := \left( \frac{\Upsilon_{\Aa} \Upsilon_{\Bb}}{\Upsilon_{\mathsf{red}(\Aa \overline{\Bb})}} \right)^{1/2},
\end{equation}
which satisfies $\mathcal{D}_{\Aa,\Bb} = O(1)$ uniformly for all $\Aa, \Bb \in \mathcal{W}^\circ$.

\begin{prop}[Phase derivative bound]\label{prop:phaseDerivative}
For $b \in \mathcal{A}$, $x \in I_b$, and distinct $\Aa, \Bb \in \mathcal{W}^\circ$ with $\Aa, \Bb \to b$,
\[
|\Phi_{\Aa, \Bb}'(x)| \asymp \mathcal{D}_{\Aa, \Bb},
\]
\end{prop}

\begin{proof}
We follow the proof of \cite[Proposition~3.2]{Soa25}. Write for all words $\Aa$ (not necessarily reduced),
\[
\gamma_{\Aa} = \pmat{a_{\Aa}}{b_{\Aa}}{c_{\Aa}}{d_{\Aa}}\in \mathrm{SL}_2(\mathbb{R}).
\]
The key observation is that for all words $\Aa$ with $\mathsf{red}(\Aa)\ne\emptyset$, we have
\[
|c_{\Aa}| \asymp \Upsilon_{\mathsf{red}(\Aa)}^{-1/2}.
\]
To see this, note first that $\gamma_{\Aa} = \gamma_{\mathsf{red}(\Aa)}$, so we may assume $\Aa$ is non-empty and reduced. We have
\[
\gamma_{\Aa}'(z) = \frac{1}{c_{\Aa}^2 (z - x_{\Aa})^2}, \quad \text{where } x_{\Aa} = -\frac{d_{\Aa}}{c_{\Aa}} = \gamma_{\overline{\Aa}}(\infty).
\]
As in the proof of Lemma~\ref{lem:bounded_distortion}, we have $|z - x_{\Aa}| \asymp 1$ for all $z \in D_b$ and $\Aa \to b$. This implies that $|\gamma_{\Aa}'(z)| \asymp \frac{1}{c_{\Aa}^2}$ for all $z \in D_b$. At the same time, Lemma~\ref{lem:boundsderivatives}(i) tells us that $|\gamma_{\Aa}'(z)| \asymp \Upsilon_{\Aa}$ for all $z \in D_b$, so we get $\Upsilon_{\Aa} \asymp \frac{1}{c_{\Aa}^2}$, which proves the claim. 

Now fix $b \in \mathcal{A}$, $x \in I_b$, and distinct $\Aa, \Bb \in \mathcal{W}^\circ$ with $\Aa, \Bb \to b$. Direct calculation yields
\begin{equation}\label{firstDer}
\Phi_{\Aa, \Bb}'(x) = 2\left( \frac{c_{\Aa}}{c_{\Aa} x + d_{\Aa}} - \frac{c_{\Bb}}{c_{\Bb} x + d_{\Bb}}  \right)
= 2 \cdot \frac{c_{\Aa} d_{\Bb} - c_{\Bb} d_{\Aa}}{(c_{\Aa} x + d_{\Aa})(c_{\Bb} x + d_{\Bb})}.
\end{equation}
Therefore,
\[
|\Phi_{\Aa, \Bb}'(x)| = 2 |c_{\Aa} d_{\Bb} - c_{\Bb} d_{\Aa}| \cdot |\gamma_{\Aa}'(x)|^{1/2} |\gamma_{\Bb}'(x)|^{1/2}.
\]
By Lemma~\ref{lem:boundsderivatives}(i),
\[
|\Phi_{\Aa, \Bb}'(x)| \asymp |c_{\Aa} d_{\Bb} - c_{\Bb} d_{\Aa}| \cdot \Upsilon_{\Aa}^{1/2} \Upsilon_{\Bb}^{1/2}.
\]
Note that
\[
\gamma_{\Aa \overline{\Bb}} = \gamma_{\Aa} \gamma_{\Bb}^{-1} =
\begin{pmatrix}
\ast & \ast \\
c_{\Aa} d_{\Bb} - c_{\Bb} d_{\Aa} & \ast
\end{pmatrix},
\]
so by the observation above,
\[
|c_{\Aa} d_{\Bb} - c_{\Bb} d_{\Aa}| \asymp \Upsilon_{\mathsf{red}(\Aa \overline{\Bb})}^{-1/2}.
\]
Going back to \eqref{firstDer}, we deduce that
\[
|\Phi_{\Aa, \Bb}'(x)| \asymp \frac{\Upsilon_{\Aa}^{1/2} \Upsilon_{\Bb}^{1/2}}{\Upsilon_{\mathsf{red}(\Aa \overline{\Bb})}^{1/2}} = \mathcal{D}_{\Aa, \Bb},
\]
as claimed.
\end{proof}

\subsection{Properties of Bergman kernels}\label{sec:Bergman}
Bergman kernels play a crucial role in this paper. We now record some of the relevant properties, referring the reader to \cite{Krantz} for more details on Bergman spaces and kernels. For a non-empty bounded open set $\Omega \subset \CC$, the Bergman (reproducing) kernel $B_\Omega(z,w)$ satisfies $B_\Omega(\cdot,w) \in H^2(\Omega)$ and
\begin{equation}\label{berg:int}
f(z) = \int_\Omega B_\Omega(z,w) f(w) \dvol(w) \quad \text{for all } f \in H^2(\Omega).
\end{equation}
Let $\{\psi_m\}_{m\in \mathbb{N}}$ be an orthonormal basis of $H^2(\Omega)$. Then we have
$$
B_{\Omega}(z,w) = \sum_{m\in \mathbb{N}} \psi_m(z) \overline{\psi_m(w)},
$$
where the series is uniformly convergent on compact subsets of $ \Omega\times \Omega. $ Note that we do not assume $\Omega$ to be connected. If $\Omega = \bigcup_{j=1}^n \Omega_j$, where $\Omega_j$ are the connected components of $\Omega$, then $H^2(\Omega) = \bigoplus_{j=1}^n H^2(\Omega_j)$ and $B_\Omega(z,w) = B_{\Omega_j}(z,w)$ if $z, w \in \Omega_j$ and $0$ if $z$ and $w$ belong to different components. We have the following basic properties:

\begin{lemma}[Properties of the Bergman kernel]\label{lem:BergmanBasicProp}
\begin{enumerate}[(i)]
\item For all $z,w\in\Omega$,
\[
|B_\Omega(z,w)|^2\le B_\Omega(z,z)B_\Omega(w,w).
\]
\item If $z \in \Omega_1 \subset \Omega_2$, then $B_{\Omega_2}(z,z) \le B_{\Omega_1}(z,z)$.
\item For the Euclidean disk $D(c,r)$ of center $c$ and radius $r$:
\[
B_{D(c,r)}(z,w) = \frac{r^2}{\pi \left( r^2 - (z - c)(\overline{w} - \overline{c}) \right)^2}.
\]
\end{enumerate}
\end{lemma}

As a consequence of the above properties, we have: 

\begin{lemma}\label{lem:UpperBoundBergman}
For all $z, w \in \Omega$,
\begin{equation}\label{UpperBoundBergmanineq}
|B_\Omega(z,w)| \ll \frac{1}{\mathrm{dist}(z, \partial\Omega) \, \mathrm{dist}(w, \partial\Omega)},
\end{equation}
where $\mathrm{dist}(z, \partial\Omega)$ denotes the Euclidean distance from $z$ to the boundary of $\Omega$. Moreover, if $\Omega$ is a finite disjoint union of Euclidean disks, then for all $z_0, w_0 \in \Omega$,
\begin{equation}\label{DerivBergmanineq}
|\partial_z B_\Omega(z_0,w_0)|, |\partial_{\overline w} B_\Omega(z_0,w_0)| \ll \frac{1}{\mathrm{dist}(z_0, \partial\Omega)^{3/2} \mathrm{dist}(w_0, \partial\Omega)^{3/2}}.
\end{equation}
The implied constants in \eqref{UpperBoundBergmanineq} and \eqref{DerivBergmanineq} are absolute.
\end{lemma}

\begin{proof}
Both bounds rest on the elementary inequality
\begin{equation}\label{eq:key_Bergman_ineq}
|r^2 - uv|^2 \ge r^2(r - |u|)(r - |v|), \quad |u|, |v| \le r.
\end{equation}
We start with the first bound. For $\Omega = D(c,r)$, Lemma~\ref{lem:BergmanBasicProp}(iii) gives 
\[
|B_\Omega(z,w)| \ll \frac{r^2}{|r^2-(z-c)(\overline{w}-\overline{c})|^2}.
\]
Applying \eqref{eq:key_Bergman_ineq} with $u = z-c$, $v = \overline{w}-\overline{c}$, and using $r-|u| = \mathrm{dist}(z,\partial\Omega)$, $r-|v| = \mathrm{dist}(w,\partial\Omega)$, yields \eqref{UpperBoundBergmanineq}. For general $\Omega$, note that $D(z,r_z), D(w,r_w) \subset \Omega$ with $r_z = \mathrm{dist}(z,\partial\Omega)$, $r_w = \mathrm{dist}(w,\partial\Omega)$, and so we can use Lemma~\ref{lem:BergmanBasicProp}(i)--(ii) to get
\[
|B_\Omega(z,w)|^2 \le B_\Omega(z,z) B_\Omega(w,w) \le B_{D(z, r_z)}(z,z) B_{D(w, r_w)}(w,w) \ll \frac{1}{r_z^2 r_w^2}.
\]
Now we prove \eqref{DerivBergmanineq}. The bound is local, so we may assume $\Omega = D(c,r)$. Differentiating Lemma~\ref{lem:BergmanBasicProp}(iii) with respect to $z$ gives
\[
\partial_z B_\Omega(z_0, w_0) = \frac{2r^2(\overline{w_0} - \overline{c})}{\pi(r^2 - (z_0-c)(\overline{w_0}-\overline{c}))^3}.
\]
Since $|\overline{w_0} - \overline{c}| = |w_0 - c| < r$, we have 
\[
|\partial_z B_\Omega(z_0,w_0)| \ll \frac{r^3}{|r^2-(z_0-c)(\overline{w_0}-\overline{c})|^3}.
\]
From \eqref{eq:key_Bergman_ineq} we have $|r^2-(z_0-c)(\overline{w_0}-\overline{c})|^3 \ge r^3\,\mathrm{dist}(z_0,\partial\Omega)^{3/2}\mathrm{dist}(w_0,\partial\Omega)^{3/2}$, yielding \eqref{DerivBergmanineq}. The bound on $\partial_{\overline w} B_\Omega$ follows by conjugate symmetry.
\end{proof}

Also helpful are the following estimates for functions in $H^2(\Omega)$:

\begin{lemma}\label{lem:bergman_estimates}
For every bounded open subset $\Omega \subset \mathbb{C}$, $f \in H^2(\Omega)$, $x_0\in \Omega$, and $k\in \mathbb{N}_0$, we have
  \[
  |f^{(k)}(x_0)| \ll \frac{2^k k!}{\mathrm{dist}(x_0, \partial \Omega)^{k+1}} \|f\|_{L^2(\Omega)},
  \]
where the implied constant is independent of $\Omega$, $x_0$, $f$, and $k$.
\end{lemma}

\begin{proof}
By the reproducing property \eqref{berg:int} and Cauchy--Schwarz, 
\[
|f(x_0)| \le \left( \int_\Omega |B_\Omega(x_0,w)|^2 \dvol(w) \right)^{1/2} \|f\|_{L^2(\Omega)} = B_\Omega(x_0, x_0)^{1/2} \|f\|_{L^2(\Omega)}.
\]
Applying Lemma~\ref{lem:UpperBoundBergman} at $z = w = x_0$ gives the bound for $k=0$. For $k\geq 1$, we use Cauchy's integral formula on a disk of radius $r = \mathrm{dist}(x_0, \partial \Omega)/2$ to get
\[
|f^{(k)}(x_0)| \le \frac{k!}{r^k} \sup_{|z-x_0|=r} |f(z)|.
\]
Since any point $z$ with $\vert z-x_0\vert = \mathrm{dist}(x_0, \partial \Omega)/2$ has distance at least $\mathrm{dist}(x_0, \partial \Omega)/2$ from the boundary of $\Omega$, this gives
\[
|f^{(k)}(x_0)| \ll \frac{k!}{r^{k+1}} \|f\|_{L^2(\Omega)} \ll \frac{2^k k!}{\mathrm{dist}(x_0, \partial \Omega)^{k+1}} \|f\|_{L^2(\Omega)},
\]
finishing the proof.
\end{proof}

\section{Proof of the main theorem}\label{sec:main_thm}
This section is devoted to the proof of Theorem~\ref{thm:main}. Fix positive parameters $h, \tau, \tau', \tau''$ to be chosen later and consider the concatenated operator
\[
\mathcal{L}_{\tau, \tau', \tau''; s} := \mathcal{L}_{\tau,s} \mathcal{L}_{\tau',s} \mathcal{L}_{\tau'',s} \colon H^2(\Omega(h)) \to H^2(\Omega(h)).
\]
This is a well-defined trace class operator by Lemma~\ref{lem:contractionInH}, provided the parameters are sufficiently small. Consider the entire function
\[
F(s) := \det(1 - \mathcal{L}_{\tau, \tau', \tau''; s}^2) = \det(1 - \mathcal{L}_{\tau, \tau', \tau''; s}) \det(1 + \mathcal{L}_{\tau, \tau', \tau''; s}).
\]
The Fredholm determinants of this operator on $H^2(D)$ and on $H^2(\Omega(h))$ agree, since both are determined by the traces of the iterates. These traces are independent of the domain, and in particular of $h$: by the Lefschetz fixed point formula, they only involve the derivatives of the composed branches at their attracting fixed points, which lie in $\Lambda$. Hence we can apply Lemma~\ref{lem:fredholmRefined} to obtain the factorization $F(s) = Z_X(s) H(s)$ for some entire function $H(s)$. Thus, since resonances of $X$ correspond to zeros of $Z_X(s)$ with the same multiplicity (except possibly in $-\mathbb{N}_0$), the function $M_X(\sigma, T)$ is bounded above by the number of zeros of $F(s)$ in the rectangular region
\[
R_{\sigma,T}:=\left\{ s \in \mathbb{C} : \sigma \le \mathrm{Re}(s) \le \delta,\ |\mathrm{Im}(s) - T| \le 1 \right\},
\]
including multiplicities. By Weyl's inequality on singular values, see \cite[Chapter~3]{Simon},
\begin{equation}\label{eq:weyl}
\log|F(s)| \le \| \mathcal{L}_{\tau, \tau', \tau''; s}^2\|_1 \le \| \mathcal{L}_{\tau, \tau', \tau''; s}\|_{\mathrm{HS}}^2,
\end{equation}
where $\|\cdot\|_1$ and $\|\cdot\|_{\mathrm{HS}}$ denote the trace norm and Hilbert--Schmidt norm on $H^2(\Omega(h))$, respectively. The key ingredient in the proof of Theorem~\ref{thm:main} is the following bound on the Hilbert--Schmidt norm of $\mathcal{L}_{\tau, \tau', \tau''; s}$ acting on the refined function space $H^2(\Omega(h))$.

\begin{prop}\label{prop:HS_bound}
Let $\epsilon>0$. For all sufficiently large $T$, there exist parameters $h,\tau,\tau',\tau''\in(0,1)$, depending only on $T$, $\epsilon$, and $\Gamma$, which are small enough for $\mathcal{L}_{\tau,\tau',\tau'';s}$ to be well defined and satisfy $\tau\tau'\tau''\le T^{-1}$, such that for all $s = \sigma + it \in \mathbb{C}$ with $\sigma > 0$ and $T/2 \le |t| \le 2T$,
\[
\| \mathcal{L}_{\tau, \tau', \tau''; s}\|_{\mathrm{HS}}^2 \ll_{\epsilon} e^{C \sigma} T^{A(\delta - \sigma) - \frac{\beta_{\mathrm{F}}}{2} + \epsilon},
\]
where $A=2+\beta_{\mathrm{F}}/\delta$.
Here $C>0$ and the implied constant are independent of $s$ and $T$.
\end{prop}

Let us first show how to deduce Theorem~\ref{thm:main} from this proposition.\\

\textbf{Deduction of Theorem~\ref{thm:main} from Proposition~\ref{prop:HS_bound}}
Fix $\sigma\in(\delta/2,\delta]$ and $\epsilon>0$. For large $T$, let $h,\tau,\tau',\tau''$ be the parameters from Proposition~\ref{prop:HS_bound}, so that $F$ depends on $T$ and $\epsilon$. Let $D_\mathbb{C}(s_0, r_1)$ and $D_\mathbb{C}(s_0, r_2)$ be concentric disks centered at $s_0 = \sigma_0 + iT$ with radii $r_2 > r_1 > 0$, chosen such that
\begin{equation}\label{concentric_jensen}
R_{\sigma,T} \subset \overline{D_\mathbb{C}(s_0, r_1)} \subset D_\mathbb{C}(s_0, r_2).
\end{equation}

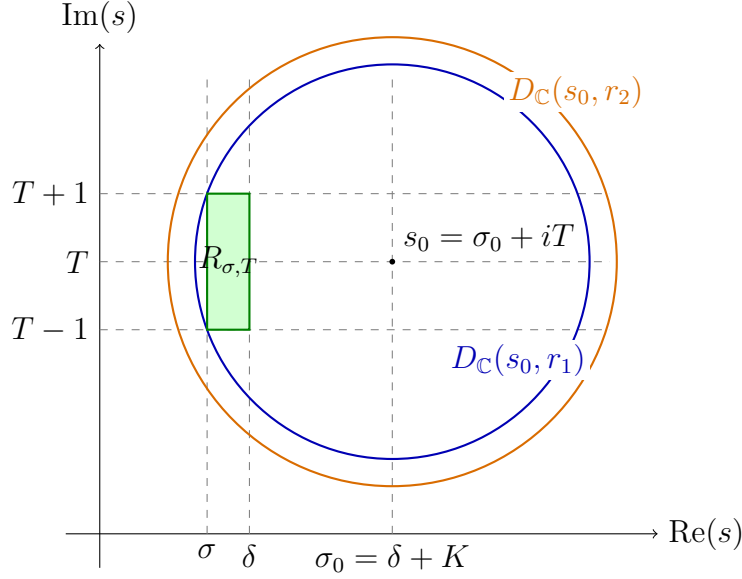
\begin{figure}[ht]
\centering
\begin{tikzpicture}[scale=0.9]
  \pgfmathsetmacro{\rOne}{2.9}
  \pgfmathsetmacro{\xZero}{4.3}
  \pgfmathsetmacro{\xSigma}{\xZero - sqrt(\rOne*\rOne - 1)}
  \pgfmathsetmacro{\yDashTop}{6.7}
  \pgfmathsetmacro{\xDashRight}{7.5}

  \draw[->] (-0.5,0) -- (8.2,0) node[right] {$\mathrm{Re}(s)$};
  \draw[->] (0,-0.5) -- (0,7.2) node[above] {$\mathrm{Im}(s)$};

  \draw[dashed,gray] (\xSigma,0) -- (\xSigma,\yDashTop);
  \draw[dashed,gray] (2.2,0) -- (2.2,\yDashTop);
  \draw[dashed,gray] (\xZero,0) -- (\xZero,\yDashTop);
  \draw[dashed,gray] (0,3.0) -- (\xDashRight,3.0);
  \draw[dashed,gray] (0,4.0) -- (\xDashRight,4.0);
  \draw[dashed,gray] (0,5.0) -- (\xDashRight,5.0);

  \node[below] at (\xSigma,0) {$\sigma$};
  \node[below] at (2.2,0) {$\delta$};
  \node[below] at (\xZero,0) {$\sigma_0 = \delta + K$};
  \node[left] at (0,3.0) {$T-1$};
  \node[left] at (0,4.0) {$T$};
  \node[left] at (0,5.0) {$T+1$};

  \coordinate (s0) at (\xZero,4.0);
  \fill (s0) circle (1.2pt);
  \node[above right] at (s0) {$s_0=\sigma_0+iT$};

  \draw[thick,blue!70!black] (s0) circle (2.9);
  \draw[thick,orange!85!black] (s0) circle (3.3);

  \filldraw[fill=green!18,draw=green!50!black,thick] (\xSigma,3.0) rectangle (2.2,5.0);
  \node at ({(\xSigma+2.2)/2},4.0) {$R_{\sigma,T}$};

  \node[blue!70!black,fill=white,inner sep=1.5pt] at (6.15,2.55) {$D_\mathbb{C}(s_0,r_1)$};
  \node[orange!85!black,fill=white,inner sep=1.5pt] at (7,6.5) {$D_\mathbb{C}(s_0,r_2)$};
\end{tikzpicture}
\caption{Geometry of the Jensen argument}
\label{fig:jensen_disks}
\end{figure}
Applying Jensen's formula gives
\[
M_X(\sigma, T) \le \frac{1}{\log(r_2 / r_1)} \left( \int_0^1 \log |F(\sigma_0 + r_2 e^{2\pi i \theta} + iT)| \, d\theta - \log |F(\sigma_0 + iT)| \right).
\]
The idea is to choose the center of the disks to be far to the right of $\delta$ so that the second term $|\log |F(\sigma_0 + iT)| |$ becomes negligibly small. Concretely, let $K\gg 1$ and choose parameters
\[
\sigma_0 := \delta + K, \quad r_1 := \sqrt{(\sigma_0-\sigma)^2 + 1}, \quad r_2 := r_1 + \frac{1}{K}.
\]
These choices guarantee \eqref{concentric_jensen}, as seen in Figure~\ref{fig:jensen_disks}. For large $K$, we have
\begin{itemize}
\item $r_1 \asymp r_2 \asymp K$, so $\log(r_2/r_1)^{-1} \ll K^2$;
\item $\sigma - O(1/K) \le \mathrm{Re}(\sigma_0 + r_2 e^{2\pi i \theta}) \le O(K)$ for all $\theta \in [0,1)$;
\item $|\mathrm{Im}(r_2 e^{2\pi i \theta})| \le r_2 = O(K)$ for all $\theta \in [0,1)$.
\end{itemize}
Combining these estimates, we see that there are constants $C_1, C_2 > 0$ such that
\[
M_X(\sigma, T) \ll K^2 \left(  \max_{\substack{ \sigma - C_1 / K \le \mathrm{Re}(s) \le C_2 K \\ |\mathrm{Im}(s)| \le C_2 K }} \log |F(s + iT)| + \left| \log |F(\delta + K + iT)| \right|\right).
\]
Recall from \eqref{eq:weyl} that $\log |F(s + iT)| \le \| \mathcal{L}_{\tau, \tau', \tau''; s + iT}\|_{\mathrm{HS}}^2.$ As long as $K \le \frac{T}{2C_2}$, we have $\mathrm{Im}(s+iT) \in [T/2, 2T]$ for all $|\mathrm{Im}(s)| \le C_2 K$. Moreover, $\mathrm{Re}(s+iT)\ge\sigma-C_1/K>0$ for $K$ large, since $\sigma>\delta/2$. Hence Proposition~\ref{prop:HS_bound} applies to the first term. For large $T$, the function $q\mapsto e^{Cq}T^{A(\delta-q)}$ is decreasing, since $A\log T>C$. Its maximum on the indicated interval occurs at $q=\sigma-C_1/K$, where $e^{Cq}=O_\sigma(1)$. This gives
\begin{equation}\label{eq:second_term_bound}
M_X(\sigma, T) \ll_{\epsilon} K^2 \left( T^{ A\left(\delta - \sigma\right) -  \frac{\beta_{\mathrm{F}}}{2} + O(1/K) + \epsilon}  + \left| \log |F(\delta + K + iT)| \right|\right).
\end{equation}
For the second term, we argue as in \cite{Soa25}. For a trace class operator $\mathcal{A}$, the Fredholm determinant admits the expansion
\begin{equation}
\det(1-\mathcal{A}) = \exp\left( -\sum_{n=1}^\infty \frac{1}{n} \mathrm{tr}(\mathcal{A}^n) \right),
\end{equation}
whenever the sum in the exponential converges absolutely; in particular, $\|\mathcal{A}\|_1 < 1$ is a sufficient condition. Taking the log-modulus then gives
\begin{equation}\label{eq:log_modulus_bound}
\vert \log \vert \det(1-\mathcal{A})\vert \vert \le \sum_{n=1}^\infty \frac{1}{n} \vert \mathrm{tr}(\mathcal{A}^n) \vert.
\end{equation}
We apply \eqref{eq:log_modulus_bound} to $\mathcal{A} = \mathcal{L}_{\tau, \tau', \tau''; s}^2$. The argument in the proof of \cite[Proposition~4.8]{NM20} carries over to the triple operator $\mathcal{L}_{\tau, \tau', \tau''; s}$ and shows that
\[
|\mathrm{tr}(\mathcal{L}_{\tau, \tau', \tau''; s}^{n})| = O((C \tau\tau'\tau'')^{n(\mathrm{Re}(s) - \delta)})
\]
for all $n \in \mathbb{N}$ and $\mathrm{Re}(s)\ge\delta+1$, with $C > 0$ and the implied constant independent of $n$, $s$, and the parameters $h, \tau, \tau', \tau''$. Inserting this into \eqref{eq:log_modulus_bound}, we deduce that for all $s$ with $\mathrm{Re}(s) = \delta + K$, and for $C \tau\tau'\tau'' < 1/2$,
\[
\vert \log \vert F(s)\vert \vert \ll \sum_{n=1}^\infty \frac{1}{n} (C \tau\tau'\tau'')^{2nK} \ll (C \tau\tau'\tau'')^{2K}.
\]
Returning to \eqref{eq:second_term_bound} and recalling that $\tau\tau'\tau''\le T^{-1}$, we get
\[
M_X(\sigma, T) \ll_{\epsilon} K^2 \left( T^{ A\left(\delta - \sigma\right) -  \frac{\beta_{\mathrm{F}}}{2} + O(1/K) + \epsilon}  + (C T^{-1})^{2K}\right).
\]
Choosing $K = \log T$, the factors $K^2$ and $T^{O(1/K)}$ are absorbed into $T^\epsilon$, while
\[
(C T^{-1})^{2K}
=O(\exp(-\log^2 T))
\]
decays faster than any polynomial in $T$. This gives
\[
M_X(\sigma, T) \ll_{\epsilon} T^{ A\left(\delta - \sigma\right) -  \frac{\beta_{\mathrm{F}}}{2} + \epsilon }.
\]
Consequently, we get
\[
m_X(\sigma) \le A(\delta-\sigma) - \frac{\beta_{\mathrm{F}}}{2}.
\]

The rest of this section is devoted to proving Proposition~\ref{prop:HS_bound}.

\subsection{Reduction to integrals over the limit set}\label{sec:HS_bound}
The proof of the above proposition relies on the following approximation result.

\begin{prop}[Holomorphic approximation of order $N$]\label{prop:approximation}
There exists $C = C(\Gamma)> 0$ such that, for every fixed integer $N\ge1$, all sufficiently small $h, \tau > 0$, all $f \in H^2(\Omega(h))$, all $s = \sigma + it \in \mathbb{C}$, and any $\epsilon > 0$, we have
\[
\Vert \mathcal L_{\tau,s}f \Vert_{L^2(\Omega(h))}^2 \ll_{N,\epsilon} e^{C (h|t| + |\sigma|)}\left(\tau^{-2\delta+2\sigma-\epsilon} h^{2-\delta} \int_{\Lambda} |f|^2\,d\mu_\delta + \tau^{N + 2\sigma} h^{-\delta} \|f\|_{L^2(\Omega(h))}^2 \right),
\]
where $\mu_\delta$ is the Ruelle--Perron--Frobenius measure from Proposition~\ref{prop:RPF}.
\end{prop}

Before proving Proposition~\ref{prop:approximation}, we need two preliminary lemmas.

\begin{lemma}\label{lem:LZ_int_bound}
For all sufficiently small $h, \tau > 0$ and for any non-negative Borel function $g \colon \Omega(h)\cap \mathbb{R} \to \mathbb{R}$, we have
\[
\int_{\Lambda} (\mathcal{L}_{\tau,\delta} g) \, d\mu_\delta \ll \log(\tau^{-1}) \int_{\Lambda} g\, d\mu_\delta.
\]
\end{lemma}

\begin{proof}[Proof of Lemma~\ref{lem:LZ_int_bound}]
By Lemma~\ref{lem:YZ}(iii), every $\Aa \in Y(\tau)$ has word length $|\Aa| \asymp \log(\tau^{-1})$, so $Y(\tau) \subset \bigcup_{n \asymp \log(\tau^{-1})} \mathcal{W}_n$. Since $g \ge 0$, this gives the pointwise bound
\[
(\mathcal{L}_{\tau,\delta} g)(x) \le \sum_{n \asymp \log(\tau^{-1})} (\mathcal{L}_{\delta}^n g)(x), \quad x \in \Lambda.
\]
Since $P(\delta) = 0$, Proposition~\ref{prop:RPF} gives $\mathcal{L}_\delta^* \mu_\delta = \mu_\delta$, hence $\int_\Lambda \mathcal{L}_\delta^n g \, d\mu_\delta = \int_\Lambda g \, d\mu_\delta$ for all $n \ge 0$. This identity holds for continuous $g$. Both sides define finite Borel measures in $g$, so it extends to all non-negative Borel functions. Integrating the pointwise bound and using that the sum contains $O(\log(\tau^{-1}))$ terms gives the result.
\end{proof}

We also need the following lemma, which is reminiscent of the Remez inequality \cite{Remez1936}.

\begin{lemma}[Remez-type inequality]\label{lem:polynomial-norming}
Fix $K\ge1$, $C_0>0$, $0<\delta<1$, and an integer $N\ge1$. There exists $C_{K,C_0,\delta,N}>0$ such that, for every probability measure $\nu$ supported on $[-K,K]$ satisfying
\begin{equation}\label{eq:measure_condition}
\nu((u-r,u+r))\le C_0r^\delta,\qquad u\in\mathbb R,\quad 0<r\le1,
\end{equation}
and every complex polynomial $p$ of degree at most $N-1$, we have
\[
\sup_{\substack{w\in\mathbb C\\|w|\le K}}|p(w)|^2 \le C_{K,C_0,\delta,N}\int |p(y)|^2\,d\nu(y).
\]
\end{lemma}

\begin{proof}
Suppose the statement is false. Then there exists a sequence of measures $\nu_j$ satisfying \eqref{eq:measure_condition} and a sequence of polynomials $p_j$ such that
\[
\sup_{|w|\le K}|p_j(w)|=1,\qquad \int|p_j|^2\,d\nu_j \to 0, \quad j\to \infty.
\]
After passing to a subsequence, we may assume that $p_j$ converges uniformly on the closed complex disk $\{|w|\le K\}$ to some polynomial $p$ with
\begin{equation}\label{eq:limit_normalization}
\sup_{|w|\le K}|p(w)|=1,
\end{equation}
and that $\nu_j$ converges weakly to some probability measure $\nu$ supported on $[-K,K]$. The first claim holds because polynomials of degree at most $N-1$ form a finite-dimensional space. The second follows from Prokhorov's theorem, since all $\nu_j$ are probability measures on the same compact interval $[-K,K]$. For every open interval $(u-r,u+r)$, the Portmanteau theorem gives
\[
\nu((u-r,u+r)) \le\liminf_{j\to\infty}\nu_j((u-r,u+r)) \le C_0r^\delta,\qquad 0<r\le1.
\]
Consequently,
\[
\nu(\{u\})\le C_0r^\delta,\qquad 0<r\le1.
\]
Letting $r\downarrow0$, using $\delta>0$, shows that $\nu$ has no atoms.

By the triangle inequality, we have
\[
\left|\int|p_j|^2\,d\nu_j-\int|p|^2\,d\nu\right| \le\sup_{y\in[-K,K]}
\bigl||p_j(y)|^2-|p(y)|^2\bigr| + \left|\int|p|^2\,d\nu_j-\int|p|^2\,d\nu\right|.
\]
The first term tends to zero by uniform convergence of $p_j$ to $p$. The second tends to zero by weak convergence, since $|p|^2$ is continuous on $[-K,K]$. Therefore, letting $j\to\infty$ in the inequality above gives
\begin{equation}
\int|p|^2\,d\nu=0.
\end{equation}
However, by \eqref{eq:limit_normalization}, $p$ is a non-zero polynomial of degree at most $N-1$, so the above equality can only be true if $\nu$ is supported on the zero set of $p$. But $\nu$ has no atoms, and a non-zero polynomial has only finitely many zeros, so this is impossible.
\end{proof}

\begin{proof}[Proof of Proposition~\ref{prop:approximation}]
Assume that $h>0$ is sufficiently small so that $\Omega(h)\subset D$. Let $D_{\ell}(h)$, $\ell = 1, \dots, N(h)$, be the disks comprising $\Omega(h)$ and let $I_\ell(h) = D_\ell(h) \cap \mathbb{R}$. Fix $\ell$ and let $b=b_\ell\in\mathcal A$ be the unique index for which $D_\ell(h)\subset D_b$. Fix $x_\ell\in\Lambda\cap I_\ell(h)$ and $\Aa\in Y_b(\tau)$. Since $(x_\ell-h,x_\ell+h)\subset I_\ell(h)$ and $\operatorname{diam}(I_\ell(h))\ll h$, Ahlfors regularity gives $\mu_\delta(I_\ell(h))\asymp h^\delta$. Set
\[
c_{\Aa,\ell}=\gamma_\Aa(x_\ell)\in\Lambda,\qquad
T_{\Aa,\ell}(x)=
\frac{\gamma_\Aa(x)-c_{\Aa,\ell}}{\tau h}.
\]
Integrating $\gamma_\Aa'$ along the segment from $x_\ell$ to $z\in D_\ell(h)$ gives
\[
|T_{\Aa,\ell}(z)| \le \sup_{w \in D_\ell(h)} |\gamma_\Aa'(w)| / (\tau h)\cdot \mathrm{diam}(D_\ell(h)) \le K,
\]
for some constant $K\ge1$ independent of $h,\tau,\ell,\Aa$, or equivalently,
\begin{equation}\label{eq:branch-image-disk}
\gamma_\Aa(D_\ell(h))
\subset \overline{D}(c_{\Aa,\ell},K\tau h).
\end{equation}
Define the probability measure
\[
\nu_{\Aa,\ell}
=(T_{\Aa,\ell})_*
\left(\frac{\mu_\delta|_{I_\ell(h)}}{\mu_\delta(I_\ell(h))}\right).
\]
It is supported on $[-K,K]$. Note that \(T_{\Aa,\ell}'(x)\asymp h^{-1}\) on $I_\ell(h)$, so the preimage within $I_\ell(h)$ of an interval of radius $r$ has length $O(rh)$. By Ahlfors regularity of $\mu_\delta$, which gives $\mu_\delta(J)\ll|J|^\delta$ for every interval $J$ with $|J|\le1$ (if $J$ meets $\Lambda$ at $y$, then $J\subset(y-|J|,y+|J|)$), we then have
\[
\nu_{\Aa,\ell}((u-r,u+r))\ll r^\delta,\qquad 0<r\le1,\quad u\in\mathbb R,
\]
uniformly in $\Aa,\ell,h,\tau$. We can therefore apply Lemma \ref{lem:polynomial-norming} to obtain for every polynomial $p$ of degree at most $N-1$,
\begin{align*}
\sup_{z\in D_\ell(h)}|p(T_{\Aa,\ell}(z))|^2 &\leq \sup_{|w|\le K}|p(w)|^2\\
&\ll_N\int |p(y)|^2\,d\nu_{\Aa,\ell}(y)\\
&= \frac{1}{\mu_\delta(I_\ell(h))} \int_{I_\ell(h)\cap \Lambda} |p(T_{\Aa,\ell}(x))|^2\,d\mu_\delta(x).
\end{align*}
Applying this to the polynomial \(q(y)=p(c_{\Aa,\ell}+\tau h y)\), and using $\mu_\delta(I_\ell(h))\asymp h^{\delta},$ we obtain
\begin{equation}\label{eq:branch-polynomial-norming}
\sup_{z\in D_\ell(h)} |p(\gamma_\Aa(z))|^2 \ll_N h^{-\delta}\int_{I_\ell(h)\cap \Lambda} |p(\gamma_\Aa(x))|^2\,d\mu_\delta(x).
\end{equation}

Let $p_{\Aa,\ell}$ be the Taylor polynomial of $f$ at $c_{\Aa,\ell}$ of degree $N-1$. This polynomial approximates $f$ near $c_{\Aa,\ell}$ up to an error that can be estimated as follows:
\[
|f(w)-p_{\Aa,\ell}(w)| \leq \sum_{k=N}^{\infty} \frac{|f^{(k)}(c_{\Aa,\ell})|}{k!} |w-c_{\Aa,\ell}|^k, \qquad w\in \mathbb{C}.
\]
Note that the point $c_{\Aa,\ell} = \gamma_\Aa(x_\ell)\in \Lambda$ has distance at least $h$ from the boundary of $\Omega(h)$. Indeed, the real diameter of the component of $\Omega(h)$ containing $c_{\Aa,\ell}$ contains $(c_{\Aa,\ell}-h,c_{\Aa,\ell}+h)$, so this component, being a disk centered on $\mathbb{R}$, contains $D(c_{\Aa,\ell},h)$. We can therefore apply Lemma \ref{lem:bergman_estimates} to get
\[
|f^{(k)}(c_{\Aa,\ell})| \leq C \frac{2^k k!}{h^{k+1}} \|f\|_{L^2(\Omega(h))}
\]
for some absolute constant $C>0$. Thus, for all $w$ with $|w-c_{\Aa,\ell}|\le K\tau h$, with $K\tau<1/4$, we can estimate the Taylor remainder as
\begin{align*}
|f(w)-p_{\Aa,\ell}(w)|
&\le C \|f\|_{L^2(\Omega(h))}\sum_{k=N}^\infty  \frac{2^k}{h^{k+1}}  |w-c_{\Aa,\ell}|^k\\
&\le C \|f\|_{L^2(\Omega(h))} \sum_{k=N}^\infty \frac{(2K\tau h)^k}{h^{k+1}}\\
&= C h^{-1} \|f\|_{L^2(\Omega(h))} \sum_{k=N}^\infty (2K\tau)^k\\
&= C h^{-1} \|f\|_{L^2(\Omega(h))} \frac{(2K\tau)^N}{1-2K\tau}\\
&\le 2 C h^{-1} \|f\|_{L^2(\Omega(h))} (2K\tau)^N.
\end{align*}
Hence, going back to \eqref{eq:branch-polynomial-norming}, replacing $p$ by the Taylor polynomial $p_{\Aa,\ell}$, and using $|p_{\Aa,\ell}|^2\le 2|f|^2+2|f-p_{\Aa,\ell}|^2$ on $\gamma_\Aa(I_\ell(h))$ together with $\mu_\delta(I_\ell(h))\ll h^\delta$, we obtain
\begin{equation}\label{eq:branch-taylor-norming}
\sup_{z\in D_\ell(h)} |f(\gamma_\Aa(z))|^2 \ll_N h^{-\delta}\int_{I_\ell(h)\cap \Lambda} |f(\gamma_\Aa(x))|^2\,d\mu_\delta(x) + h^{-2} \tau^{2N} \|f\|_{L^2(\Omega(h))}^2.
\end{equation}
Note that the factor $(2K)^{2N}$ is absorbed into the implicit constant $\ll_N$.

By Cauchy--Schwarz and the bounds in
Lemmas~\ref{lem:derivative_bounds_omega} and~\ref{lem:YZ},
\[ 
|\mathcal L_{\tau,s}f(z)|^2 \ll E\tau^{-\delta+2\sigma} \sum_{\Aa\in Y_{b_\ell}(\tau)}|f(\gamma_\Aa(z))|^2, \qquad z\in D_\ell(h), 
\]
where $E = e^{C (h|t| + |\sigma|)}$ for some constant $C = C(\Gamma)>0$. Using \eqref{eq:branch-taylor-norming}, we get for all $z\in D_\ell(h)$
\begin{align*}
|\mathcal L_{\tau,s}f(z)|^2
&\ll_N E\tau^{-\delta+2\sigma}h^{-\delta}
\int_{I_\ell(h)\cap\Lambda}\sum_{\Aa\in Y_{b_\ell}(\tau)}|f(\gamma_\Aa(x))|^2\,d\mu_\delta(x)\\
&\quad+E\tau^{-2\delta+2\sigma+2N}h^{-2}\|f\|_{L^2(\Omega(h))}^2\\
&\ll_N E\tau^{-2\delta+2\sigma}h^{-\delta}
\int_{I_\ell(h)\cap\Lambda}\mathcal L_{\tau,\delta}|f|^2\,d\mu_\delta\\
&\quad+E\tau^{-2\delta+2\sigma+2N}h^{-2}\|f\|_{L^2(\Omega(h))}^2.
\end{align*}
Integrating this over $z\in D_\ell(h)$ with respect to the Lebesgue measure on the complex plane and noting that $\vol(D_{\ell}(h))\ll h^2$, yields
$$
\Vert \mathcal L_{\tau,s}f \Vert_{L^2(D_\ell(h))}^2 \ll_N  E\tau^{-2\delta+2\sigma} h^{2-\delta} \int_{I_\ell(h)\cap \Lambda} \mathcal{L}_{\tau,\delta} |f|^2\,d\mu_\delta +  E\tau^{-2\delta+2\sigma+2N}\|f\|_{L^2(\Omega(h))}^2.
$$
Summing this over all $\ell$ gives
$$
\Vert \mathcal L_{\tau,s}f \Vert_{L^2(\Omega(h))}^2 \ll_N  E\tau^{-2\delta+2\sigma} h^{2-\delta} \int_{\Lambda} \mathcal{L}_{\tau,\delta} |f|^2\,d\mu_\delta +  E\tau^{-2\delta+2\sigma+2N} h^{-\delta} \|f\|_{L^2(\Omega(h))}^2.
$$
Finally, Lemma~\ref{lem:LZ_int_bound} gives for any $\epsilon > 0$,
$$
\int_{\Lambda} \mathcal{L}_{\tau,\delta} |f|^2\,d\mu_\delta \ll_\epsilon \tau^{-\epsilon} \int_{\Lambda} |f|^2\,d\mu_\delta.
$$
Inserting this into the previous estimate gives
$$
\Vert \mathcal L_{\tau,s}f \Vert_{L^2(\Omega(h))}^2 \ll_{N,\epsilon}    E\tau^{-2\delta+2\sigma-\epsilon} h^{2-\delta} \int_{\Lambda} |f|^2\,d\mu_\delta +  E\tau^{-2\delta+2\sigma+2N} h^{-\delta} \|f\|_{L^2(\Omega(h))}^2.
$$
Replacing $N$ with $\lceil\frac{1}{2}(N + 2\delta)\rceil$ gives the simpler estimate
 $$
\Vert \mathcal L_{\tau,s}f \Vert_{L^2(\Omega(h))}^2 \ll_{N,\epsilon} E\tau^{-2\delta+2\sigma-\epsilon} h^{2-\delta} \int_{\Lambda} |f|^2\,d\mu_\delta +  E\tau^{N + 2\sigma} h^{-\delta} \|f\|_{L^2(\Omega(h))}^2,
$$
which completes the proof.
\end{proof}

\subsection{Applying the Fourier decay bound}\label{sec:BD_bound}
The last major technical ingredient in the proof of Proposition~\ref{prop:HS_bound} is the following estimate:

\begin{prop}[Integral bound]\label{prop:integral_bound}
Let $\{\psi_m\}$ be an orthonormal basis of $H^2(\Omega(h))$, with $h>0$ sufficiently small. Set
\begin{equation*}
\tau'' := C_{\mathrm{sep}} h, \quad \tau' := T^{-\beta_{\mathrm{F}}/(2\delta)},
\end{equation*}
where $C_{\mathrm{sep}} > 0$ is the constant from Lemma \ref{lem:separation} and $\beta_{\mathrm{F}} > 0$ is the Fourier decay exponent defined after Theorem~\ref{thm:BD_fourier_decay}. Then for all $s = \sigma + it \in \mathbb{C}$ with $T/2 \le |t| \le 2T$, $T$ sufficiently large, we have
\[
\sum_{m\in \mathbb{N}} \int_{\Lambda} |\mathcal{L}_{\tau', \tau''; s} \psi_m|^2 \, d\mu_\delta \ll e^{C|\sigma|} h^{-2} (\tau'\tau'')^{-\delta + 2\sigma},
\]
where $\mathcal{L}_{\tau', \tau''; s} := \mathcal{L}_{\tau',s} \mathcal{L}_{\tau'',s}.$
\end{prop}

It is here that we apply the Bourgain--Dyatlov Fourier decay bound from Theorem~\ref{thm:BD_fourier_decay}. To apply this result, we first localize to the intervals $I_b$ and replace the Patterson--Sullivan measure $\mu$ by the Ruelle--Perron--Frobenius measure $\mu_\delta$ from Proposition~\ref{prop:RPF}.

\begin{lemma}[Localized Bourgain--Dyatlov estimate]\label{lem:localized_BD}
For $b\in\mathcal A$, let $\varphi\in C^2(I_b;\mathbb R)$ such that
\[
\|\varphi\|_{C^2(I_b)}\le C_0,
\qquad
\inf_{\Lambda\cap I_b}|\varphi'|\ge C_0^{-1}.
\]
Then, for every $g\in C^1(I_b;\mathbb C)$,
\[
\left|\int_{\Lambda\cap I_b}e^{i\xi\varphi(x)}g(x)\,d\mu_\delta(x)\right|
\ll_{C_0}|\xi|^{-\beta_{\mathrm{F}}/2}\|g\|_{C^1(I_b)},
\qquad |\xi|>1.
\]
\end{lemma}

\begin{proof}
Choose fixed compact intervals $J_b\subset I_b\subset J_b'$ such that $\Lambda\cap I_b\subset\operatorname{int}(J_b)$ and
\[
\Lambda\cap\bigl(J_b'\setminus\operatorname{int}(J_b)\bigr)=\emptyset.
\]
Such intervals exist because the Schottky disks have pairwise disjoint closures. Fix a cut-off function $\chi_b\in C_c^\infty(J_b')$ such that $\chi_b\equiv 1$ on a neighborhood of $J_b$ and $\chi_b\equiv 0$ outside $J_b'$. By standard extension from a fixed interval, there are extensions $\varphi_0\in C^2(\mathbb R;\mathbb R)$ and $g_0\in C^1(\mathbb R;\mathbb C)$ of $\varphi|_{J_b}$ and $g|_{J_b}$ satisfying
\[
\|\varphi_0\|_{C^2(\mathbb R)}\ll C_0,
\qquad
\|g_0\|_{C^1(\mathbb R)}\ll\|g\|_{C^1(I_b)}.
\]
We may assume that $C_0\ge1$. Fix $x_b\in J_b$ and a fixed function $\chi\in C_c^\infty(\mathbb R)$ equal to $1$ on a neighborhood of $\Lambda\cup J_b'$, and set
\[
\widetilde\varphi(x):=\chi(x)\Bigl(\chi_b(x)\varphi_0(x)+(1-\chi_b(x))\bigl(\varphi_0(x_b)+x-x_b\bigr)\Bigr),
\qquad \widetilde g(x):=\chi_b(x)g_0(x).
\]
We have $\widetilde\varphi=\varphi$ and $\widetilde g=g$ on $\Lambda\cap I_b$. We need to verify that these functions satisfy the assumptions of Theorem~\ref{thm:BD_fourier_decay}. The bounds
\[
\|\widetilde\varphi\|_{C^2(\mathbb R)}\ll C_0,
\qquad
\|\widetilde g\|_{C^1(\mathbb R)}\ll\|g\|_{C^1(I_b)},
\]
are immediate from the construction. The derivative lower bound follows from the fact that $\widetilde\varphi'=\varphi'$ on $\Lambda\cap I_b$ and $\widetilde\varphi'=1$ on $\Lambda\setminus I_b$, since this latter set is disjoint from $J_b'$. Hence, we have
\[
\inf_\Lambda|\widetilde\varphi'|\ge C_0^{-1}.
\]
Finally, the measures $\mu_\delta$ and $\mu$ are related as follows: $d\mu_\delta=C(1+x^2)^\delta\,d\mu$; see \cite[Lemma~4.7]{JNS20}. Thus, applying Theorem~\ref{thm:BD_fourier_decay} with exponent $\beta_{\mathrm{F}}$ to $\widetilde\varphi$ and $(1+x^2)^\delta\widetilde g$, and then scaling the amplitude, proves the claim.
\end{proof}

We now apply this localized estimate to our phase functions $\Phi_{\Aa,\Bb}$ defined in \eqref{phase}.

\begin{cor}\label{cor:BD_fourier_decay_cor}
Let $b \in \mathcal{A}$ and let $\Aa, \Bb \in \mathcal{W}^\circ$ be distinct words with $\Aa, \Bb \to b$. Then for every $g \in C^1(I_b; \mathbb{C})$, we have
\begin{equation}\label{eq:BD_fourier_decay_cor}
\left| \int_{\Lambda \cap I_b} e^{i\xi \Phi_{\Aa, \Bb}(x)} g(x) \, d\mu_\delta(x) \right| \le C  ( \mathcal{D}_{\Aa, \Bb} |\xi| )^{-\beta_{\mathrm{F}}/2} \Vert g \Vert_{C^1(I_b)}
\qquad \text{for all } |\xi| > 1,
\end{equation}
where $C > 0$ depends only on $X$.
\end{cor}

\begin{proof}
Fix $x_0 \in I_b$ and define for all $x \in I_b$ the real-valued function
\[
\varphi_{\Aa, \Bb}(x) := \frac{\Phi_{\Aa, \Bb}(x) - \Phi_{\Aa, \Bb}(x_0)}{\mathcal{D}_{\Aa, \Bb}}.
\]
Note that
\[
\left|\int_{\Lambda\cap I_b}e^{i\xi\Phi_{\Aa,\Bb}(x)}g(x)\,d\mu_\delta(x)\right| = \left|\int_{\Lambda\cap I_b}e^{i \mathcal{D}_{\Aa, \Bb} \xi\varphi_{\Aa,\Bb}(x)}g(x)\,d\mu_\delta(x)\right|.
\]
If we can show that $\varphi = \varphi_{\Aa, \Bb}$ satisfies the assumptions of Lemma~\ref{lem:localized_BD} with some uniform constant $C_0 >0$, then we can apply that lemma to obtain the desired Fourier decay estimate, as long as $\mathcal D_{\Aa,\Bb}|\xi|>1$. For $\mathcal D_{\Aa,\Bb}|\xi|\le 1$, the bound is trivial.

By Proposition~\ref{prop:phaseDerivative},
\[
|\varphi_{\Aa,\Bb}'|\asymp1 \quad\text{on }I_b.
\]
In particular, 
\[
\inf_{\Lambda\cap I_b}|\varphi_{\Aa, \Bb}'|\gg 1.
\]
Using the mean-value theorem, we also have
\[
|\varphi_{\Aa, \Bb}(x) | \asymp |x - x_0| \ll 1 \quad\text{for all } x \in I_b.
\]
It remains to verify that the second derivative $\varphi_{\Aa,\Bb}'' = \Phi_{\Aa,\Bb}'' / \mathcal{D}_{\Aa, \Bb}$ is uniformly bounded. Differentiating \eqref{eq:log_der} gives
\begin{equation}\label{eq:second_derivative_Phi}
\Phi_{\Aa,\Bb}''(x) =-\Phi_{\Aa,\Bb}'(x)\left(\frac{1}{x-x_{\Aa}}+\frac{1}{x-x_{\Bb}}\right),
\end{equation}
where $x_{\Aa}:=\gamma_{\Aa}^{-1}(\infty)$ and $x_{\Bb}:=\gamma_{\Bb}^{-1}(\infty)$. Since $x_{\Aa}$ and $x_{\Bb}$ lie in Schottky disks distinct from $D_b$, we have $\vert x - x_{\Aa} \vert \asymp 1$ and $\vert x - x_{\Bb} \vert \asymp 1$ for all $x \in I_b$. Using this in \eqref{eq:second_derivative_Phi} yields
\[
|\Phi_{\Aa,\Bb}''(x)|\ll|\Phi_{\Aa,\Bb}'(x)|\quad\text{on }I_b.
\]
Combining these estimates shows that
\[
\|\varphi_{\Aa, \Bb}\|_{C^2(I_b)}\ll 1
\]
as desired.
\end{proof}

We are now ready to complete the proof of Proposition \ref{prop:integral_bound}.

\begin{proof}[Proof of Proposition \ref{prop:integral_bound}]
Fix an orthonormal basis $\{\psi_m\}_{m\in \mathbb{N}}$ of $H^2(\Omega(h))$, and recall that the Bergman kernel of $H^2(\Omega(h))$ can be expressed for all $z, w \in \Omega(h)$ as the absolutely convergent series
\[
B_{\Omega(h)}(z,w) = \sum_{m\in \mathbb{N}} \psi_m(z) \overline{\psi_m(w)}.
\]
By the definition of the refined operators, we have for $z\in\Omega_b(h)$, $b\in\mathcal{A}$,
\[
\mathcal{L}_{\tau', \tau''; s}f(z) = \sum_{\Aa_2\in Y_b(\tau')} \sum_{\Aa_1\in Y_{\mathrm{first}(\Aa_2)}(\tau'')} \gamma_{\Aa_1\Aa_2}'(z)^s f(\gamma_{\Aa_1\Aa_2}(z)),
\]
where
\[
S_b:=\{(\Aa_1,\Aa_2):\Aa_2\in Y_b(\tau'),\ \Aa_1\in Y_{\mathrm{first}(\Aa_2)}(\tau'')\}.
\]
Here, $\mathrm{first}(\Aa_2)$ denotes the first letter of the word $\Aa_2$. Therefore, we have the following expansion for every $x \in \Lambda\cap I_b$:
\begin{equation}
\begin{aligned}
\sum_m|\mathcal L_{\tau',\tau'';s}\psi_m(x)|^2
&=\sum_{(\Aa_1,\Aa_2),(\Bb_1,\Bb_2)\in S_b}
\gamma_{\Aa_1\Aa_2}'(x)^s\overline{\gamma_{\Bb_1\Bb_2}'(x)^s}\\
&\quad\times B_{\Omega(h)}(\gamma_{\Aa_1\Aa_2}(x),\gamma_{\Bb_1\Bb_2}(x)).
\end{aligned}
\end{equation}
Define the function
\[
g_{\Aa,\Bb}(z;s) := \gamma'_\Aa(z)^s \overline{\gamma'_\Bb(z)^s} B_{\Omega(h)}(\gamma_\Aa(z), \gamma_\Bb(z)).
\]
Hence, we can write
\begin{equation}\label{eq:explicit_formula}
\int_{\Lambda}\sum_{m\in \mathbb{N}}|\mathcal{L}_{\tau', \tau''; s} \psi_m(x)|^2 d\mu_\delta = \sum_{b\in\mathcal A} \sum_{(\Aa_1,\Aa_2),(\Bb_1,\Bb_2)\in S_b} \int_{\Lambda\cap I_b} g_{\Aa_1\Aa_2,\Bb_1\Bb_2}(x;s) \, d\mu_\delta(x).
\end{equation}

For sufficiently small scales, all composite branches $\Aa=\Aa_1\Aa_2$ with $(\Aa_1,\Aa_2)\in S_b$ have length at least $N_0$. Hence Lemma~\ref{lem:contractionInH} gives, for $z\in\Omega_b(h)$,
\[
\dist(\gamma_\Aa(z), \partial \Omega(h)) \gg h, 
\]
and similarly for $\gamma_\Bb(z)$. Thus Lemma~\ref{lem:UpperBoundBergman} gives $|B_{\Omega(h)}(\gamma_\Aa(z),\gamma_\Bb(z))|\ll h^{-2}$. Combining this with Lemma~\ref{lem:derivative_bounds_omega} gives
\begin{equation}\label{eq:pointwise_g_bound}
|g_{\Aa,\Bb}(z;s)| \ll h^{-2} |\gamma'_\Aa(z)^s| |\gamma'_\Bb(z)^s| \ll e^{C (h|t| + \vert \sigma \vert)} h^{-2} (\tau'\tau'')^{2\sigma}.
\end{equation}
When $x$ is real, we have $|g_{\Aa,\Bb}(x;s)|=|g_{\Aa,\Bb}(x;\sigma)|,$ in which case we can drop the exponential factor $e^{Ch|t|}.$

Recall that $\tau'' = C_{\mathrm{sep}} h $, where $C_{\mathrm{sep}}$ is the constant from Lemma~\ref{lem:separation}. This lemma, applied with $z_1=\gamma_{\Aa_2}(z)$ and $z_2=\gamma_{\Bb_2}(z)$, gives $\Aa_1\sim_{\mathrm{pre}}\Bb_1$ or otherwise $\gamma_{\Aa_1\Aa_2}(z)$ and $\gamma_{\Bb_1\Bb_2}(z)$ belong to different connected components of $\Omega(h)$. Note that the Bergman kernel $B_{\Omega(h)}(z_1, z_2)$ vanishes if $z_1$ and $z_2$ are in different connected components of $\Omega(h)$. We conclude the sum on the right of \eqref{eq:explicit_formula} reduces to terms where $\Aa_1$ and $\Bb_1$ are prefix-comparable. 

Let $F$ be the finite set from Lemma~\ref{lem:boundedTransitions}. It follows from this lemma that every pair $(\Aa,\Bb) = (\Aa_1\Aa_2,\Bb_1\Bb_2)$ in this sum with $\Aa_1\sim_{\mathrm{pre}}\Bb_1$ belongs to one of the following three classes:
\begin{align*}
Q_b^{(1)} &:= \{ (\Aa_1\Aa_2,\Bb_1\Bb_2): (\Aa_1,\Aa_2),(\Bb_1,\Bb_2)\in S_b,\ \Aa_1\sim_{\mathrm{pre}}\Bb_1,\ \Aa_2\sim_{\mathrm{suf}}\Bb_2\},\\
Q_b^{(2)} &:= \{ (\Cc\Aa_2,\Cc\Bb_2): (\Cc,\Aa_2),(\Cc,\Bb_2)\in S_b,\ \Aa_2\nsim_{\mathrm{suf}}\Bb_2\},\\
Q_b^{(3)} &:= \{ (\Cc\Aa_2,\Cc\mathbf{v}\Bb_2): \mathbf{v}\in F,\ (\Cc,\Aa_2),(\Cc\mathbf{v},\Bb_2)\in S_b,\ \Aa_2\nsim_{\mathrm{suf}}\Bb_2\}\\
&\quad\cup\{ (\Cc\mathbf{v}\Aa_2,\Cc\Bb_2): \mathbf{v}\in F,\ (\Cc\mathbf{v},\Aa_2),(\Cc,\Bb_2)\in S_b,\ \Aa_2\nsim_{\mathrm{suf}}\Bb_2\}.
\end{align*}
By the triangle inequality, we have
\[
\sum_{m\in \mathbb{N}}\int_{\Lambda}|\mathcal{L}_{\tau', \tau''; s} \psi_m|^2 d\mu_\delta\le \Sigma_1 + \Sigma_2 + \Sigma_3,
\]
where
\[
\Sigma_j := \sum_{b\in\mathcal A}\sum_{(\Aa,\Bb)\in Q_b^{(j)}} \left| \int_{\Lambda \cap I_b} g_{\Aa,\Bb}(x;s) d\mu_\delta(x) \right|, \quad j\in\{1,2,3\}.
\]
We will bound each $\Sigma_j$ separately. By Lemma~\ref{lem:boundedTransitions} and Lemma~\ref{lem:YZ}, we have
\[
|Q_b^{(1)}| \ll |S_b| \ll |Y(\tau')|| Y(\tau'')| \ll (\tau'\tau'')^{-\delta}.
\]
Together with the pointwise estimate in \eqref{eq:pointwise_g_bound}, we obtain
\[
\Sigma_1 \ll e^{C|\sigma|} h^{-2}(\tau'\tau'')^{-\delta + 2\sigma}.
\]
For $(\Aa,\Bb)\in Q_b^{(2)}\cup Q_b^{(3)}$, write $g_{\Aa,\Bb}(x;s) = g_{\Aa,\Bb}(x;\sigma) e^{-it \Phi_{\Aa,\Bb}(x)}$. By Proposition \ref{prop:phaseDerivative}, we have $|\Phi_{\Aa,\Bb}'| \asymp \mathcal{D}_{\Aa, \Bb}$ on $I_b$. Hence, applying Corollary \ref{cor:BD_fourier_decay_cor}, we obtain
\begin{equation}\label{eq:BD_bound_applied}
\left| \int_{\Lambda \cap I_b} g_{\Aa,\Bb}(x;\sigma) e^{-it \Phi_{\Aa,\Bb}(x)} d\mu_\delta(x) \right| \ll \|g_{\Aa,\Bb}(\cdot;\sigma)\|_{C^1(I_b)} (\mathcal{D}_{\Aa, \Bb} |t|)^{-\beta_{\mathrm{F}}/2}.
\end{equation}
Before estimating this norm, fix $x_b\in\Lambda\cap I_b$. For every composite branch $\Aa$ occurring here, $\gamma_\Aa(x_b)\in\Lambda$ and the derivative bounds give
\[
|\gamma_\Aa(x)-\gamma_\Aa(x_b)|\ll\tau'\tau''\asymp h\tau',\qquad x\in I_b.
\]
Since $\tau'\to0$, the entire interval $\gamma_\Aa(I_b)$ lies in $\Omega(h/2)$ for sufficiently large $T$ and has distance $\gg h$ from $\partial\Omega(h)$; the same holds for $\Bb$. Thus $g_{\Aa,\Bb}(\cdot;\sigma)$ is well-defined and $C^1$ on $I_b$, and the pointwise bound in \eqref{eq:pointwise_g_bound} holds there. It gives
\[
\Vert g_{\Aa,\Bb}(\cdot;\sigma)\Vert_{C^0(I_b)} \ll e^{C|\sigma|} h^{-2}(\tau'\tau'')^{2\sigma}.
\]
To bound the maximum of the derivative, write
\begin{align*}
\frac{d}{dx}g_{\Aa,\Bb}(x;\sigma)
&=\frac{d}{dx}\left(\gamma'_\Aa(x)^\sigma\gamma'_\Bb(x)^\sigma\right)
B_{\Omega(h)}(\gamma_\Aa(x),\gamma_\Bb(x))\\
&\quad+\gamma'_\Aa(x)^\sigma\gamma'_\Bb(x)^\sigma
\frac{d}{dx}B_{\Omega(h)}(\gamma_\Aa(x),\gamma_\Bb(x)).
\end{align*}
Since $|\gamma_\Aa''(x)/\gamma_\Aa'(x)| \ll 1$ for $x\in I_b$, we get $ \frac{d}{dx} \left( \gamma'_\Aa(x)^\sigma \gamma'_\Bb(x)^\sigma \right) \ll e^{C|\sigma|} (\tau'\tau'')^{2\sigma}$. Moreover, writing
$$
\begin{aligned}
\frac{d}{dx} B_{\Omega(h)}(\gamma_\Aa(x), \gamma_\Bb(x))
&= \partial_z B_{\Omega(h)}(\gamma_\Aa(x), \gamma_\Bb(x)) \cdot \gamma_\Aa'(x)\\
&\quad+\partial_{\overline w} B_{\Omega(h)}(\gamma_\Aa(x), \gamma_\Bb(x)) \cdot \gamma_\Bb'(x),
\end{aligned}
$$
we obtain from Lemma~\ref{lem:UpperBoundBergman},
$$
\left|\frac{d}{dx}B_{\Omega(h)}(\gamma_\Aa(x),\gamma_\Bb(x))\right| \ll h^{-3}\tau'\tau''.
$$
Combining the above estimates, we get $ \sup_{x\in I_b} \left| \frac{d}{dx} g_{\Aa,\Bb}(x;\sigma) \right| \ll e^{C|\sigma|} h^{-2}(\tau'\tau'')^{2\sigma}(1+h^{-1}\tau'\tau'')$. Consequently,
\begin{equation*}
\|g_{\Aa,\Bb}(\cdot;\sigma)\|_{C^1(I_b)} \ll e^{C|\sigma|} h^{-2}(\tau'\tau'')^{2\sigma}\bigl(1+h^{-1}\tau'\tau''\bigr).
\end{equation*}
Since $\tau''\asymp h$ and $\tau'\ll 1$, this simplifies to
\begin{equation*}
\|g_{\Aa,\Bb}(\cdot;\sigma)\|_{C^1(I_b)} \ll e^{C|\sigma|} h^{-2}(\tau'\tau'')^{2\sigma}.
\end{equation*}

Now we estimate the quantity $\mathcal{D}_{\Aa, \Bb} $. Suppose first that $(\Aa,\Bb)=(\Cc\Aa_2,\Cc\Bb_2)\in Q_b^{(2)}$. Recall that
\begin{itemize}
\item $\Cc\in Y(\tau'')$ and $\Aa_2,\Bb_2\in Y(\tau')$,
\item $\Cc\to \Aa_2,\Bb_2$, and
\item $\Aa_2\nsim_{\mathrm{suf}}\Bb_2$.
\end{itemize}
Let $\Dd$ be the largest common suffix of $\Aa_2$ and $\Bb_2$ and write $\Aa_2 = \Ee \Dd$ and $\Bb_2 = \Ff \Dd$. Since neither $\Aa_2$ nor $\Bb_2$ is a suffix of the other, both $\Ee$ and $\Ff$ are non-empty words with distinct ending letters. This gives
\[
\mathsf{red}(\Aa\overline{\Bb})
=\Cc\Ee\overline{\Ff}\overline{\Cc}
=\Cc\,\mathsf{red}(\Aa_2\overline{\Bb_2})\,\overline{\Cc},
\]
and $\Cc\to \mathsf{red}(\Aa_2\overline{\Bb_2})\to\overline{\Cc}.$ Repeated applications of Lemma~\ref{lem:boundsderivatives} now give
\[
\mathcal{D}_{\Aa, \Bb}
= \left( \frac{\Upsilon_{\Aa} \Upsilon_{\Bb}}{\Upsilon_{\mathsf{red}(\Aa \overline{\Bb})}} \right)^{1/2}
\asymp \left( \frac{\Upsilon_{\Aa_2} \Upsilon_{\Bb_2}}{\Upsilon_{\mathsf{red}(\Aa_2 \overline{\Bb_2})}} \right)^{1/2}
\asymp \frac{\tau'}{\Upsilon_{\mathsf{red}(\Aa_2 \overline{\Bb_2})}^{1/2}}.
\]
Thus, \eqref{eq:BD_bound_applied} and the preceding $C^1$ bound imply
\begin{align*}
\left|\int_{\Lambda\cap I_b}g_{\Aa,\Bb}(x;s)\,d\mu_\delta(x)\right|
&\ll e^{C|\sigma|}h^{-2}(\tau'\tau'')^{2\sigma}
\left(\frac{\tau'|t|}{\Upsilon_{\mathsf{red}(\Aa_2\overline{\Bb_2})}^{1/2}}\right)^{-\beta_{\mathrm{F}}/2}\\
&=e^{C|\sigma|}h^{-2}(\tau'\tau'')^{2\sigma}
(\tau')^{-\beta_{\mathrm{F}}/2}|t|^{-\beta_{\mathrm{F}}/2}
\Upsilon_{\mathsf{red}(\Aa_2\overline{\Bb_2})}^{\beta_{\mathrm{F}}/4}.
\end{align*}
For fixed $\Aa_2,\Bb_2$, the common prefix $\Cc$ ranges over a subset of $Y(\tau'')$. Hence, using $|Y(\tau'')|\ll(\tau'')^{-\delta}$, we obtain
\[
\Sigma_2 \ll e^{C |\sigma|} h^{-2} (\tau'\tau'')^{2\sigma} (\tau'')^{-\delta} (\tau')^{-\beta_{\mathrm{F}}/2} |t|^{-\beta_{\mathrm{F}}/2}\sum_{\Aa_2,\Bb_2\in Y(\tau')} \Upsilon_{\mathsf{red}(\Aa_2 \overline{\Bb_2})}^{\beta_{\mathrm{F}}/4}.
\]
To estimate the sum on the right, we apply Lemma~\ref{lem:sum_of_Upsilons} below with $\alpha=\beta_{\mathrm{F}}/4<\delta/2$. We obtain
\begin{equation}\label{eq:sigma_2}
\Sigma_2 \ll e^{C |\sigma|} h^{-2} (\tau'\tau'')^{2\sigma} (\tau'')^{-\delta} (\tau')^{-2\delta} |t|^{-\beta_{\mathrm{F}}/2}.
\end{equation}

Finally, we consider the set $Q_b^{(3)}$. Suppose without loss of generality that $(\Aa,\Bb)$ belongs to the first part of $Q_b^{(3)}$. The second part of $Q_b^{(3)}$ is handled similarly. Then there exists a non-empty word $\mathbf{v}\in F$ such that $(\Aa,\Bb)  = (\Cc\Aa_2,\Cc\mathbf{v}\Bb_2)$. Put $\widetilde{\Bb}_2:=\mathbf{v}\Bb_2$. We claim that $\Aa_2\nsim_{\mathrm{suf}}\widetilde{\Bb}_2$. Suppose not. If $\widetilde{\Bb}_2$ is a suffix of $\Aa_2$, then $\Bb_2$ is a suffix of $\Aa_2$, contradicting $\Aa_2\nsim_{\mathrm{suf}}\Bb_2$. If, on the other hand, $\Aa_2$ is a suffix of $\widetilde{\Bb}_2$, then $\Aa_2\sim_{\mathrm{suf}}\Bb_2$, again a contradiction.

Hence, we can apply the same argument as before to $\Aa_2$ and $\widetilde{\Bb}_2$ instead of $\Bb_2$ to deduce that
\[
\mathcal{D}_{\Cc\Aa_2,\Cc\mathbf{v}\Bb_2}
\asymp_F \frac{\tau'}{\Upsilon_{\mathsf{red}(\Aa_2\overline{\Bb_2}\overline{\mathbf{v}})}^{1/2}}.
\]
For $\mathbf{v}\in F$, let
\[
\mathcal{V}_{\mathbf{v}}:=\{\mathbf{v}\Bb_2:\Bb_2\in Y(\tau'),\ \mathbf{v}\to\Bb_2\}.
\]
The same chain of inequalities as before shows that
\[
\Sigma_3 \ll e^{C |\sigma|} h^{-2} (\tau'\tau'')^{2\sigma} (\tau'')^{-\delta} (\tau')^{-\beta_{\mathrm{F}}/2} |t|^{-\beta_{\mathrm{F}}/2}
\sum_{\mathbf{v}\in F}\sum_{\Aa_2\in Y(\tau'),\,\Bb_2\in\mathcal{V}_{\mathbf{v}}}
\Upsilon_{\mathsf{red}(\Aa_2\overline{\Bb_2})}^{\beta_{\mathrm{F}}/4}.
\]
Since $F$ is finite, for each fixed $\mathbf{v}\in F$, the set $\mathcal{V}_{\mathbf{v}}$ consists of words whose $\Upsilon$-size is comparable to $\tau'$. We can apply Lemma~\ref{lem:sum_of_Upsilons} below to $\mathcal{U} = Y(\tau')$ and $\mathcal{V} = \mathcal{V}_{\mathbf{v}}$ and then sum over $F$ to get
\begin{equation}\label{eq:sigma_3}
\Sigma_3 \ll e^{C |\sigma|} h^{-2} (\tau'\tau'')^{2\sigma} (\tau'')^{-\delta} (\tau')^{-2\delta} |t|^{-\beta_{\mathrm{F}}/2}.
\end{equation}

Combining the estimates for $\Sigma_1$, $\Sigma_2$, and $\Sigma_3$ yields
\begin{equation*}
\sum_{m\in \mathbb{N}}\int_{\Lambda}|\mathcal{L}_{\tau', \tau''; s} \psi_m|^2 d\mu_\delta
\ll e^{C |\sigma|} h^{-2} (\tau'\tau'')^{-\delta + 2\sigma} \left(1 + (\tau')^{-\delta} |t|^{-\beta_{\mathrm{F}}/2} \right).
\end{equation*}
Finally, choosing $\tau' = T^{-\beta_{\mathrm{F}}/(2\delta)}$ and recalling that $T/2 \le |t| \le 2T$, we see that the term in parentheses is $O(1)$, and hence we get the desired bound.
\end{proof}

It remains to prove the following estimate we used in the proof above, namely:

\begin{lemma}\label{lem:sum_of_Upsilons}
Let $C_0\ge1$, and let $\mathcal{U},\mathcal{V}\subset\mathcal{W}$ be finite sets such that for all $\Aa\in\mathcal{U}$ and $\Bb\in\mathcal{V}$,
\[
C_0^{-1}\tau\le\Upsilon_{\Aa}\le C_0\tau,
\qquad
C_0^{-1}\tau\le\Upsilon_{\Bb}\le C_0\tau.
\]
Then, for all $0<\tau\le1$, $\alpha\in\mathbb{R}$, and $\epsilon>0$,
\[
\sum_{\Aa\in\mathcal{U},\,\Bb\in\mathcal{V}}
\Upsilon_{\mathsf{red}(\Aa\overline{\Bb})}^{\alpha}
\ll_{\alpha,\epsilon,C_0} \tau^{-\delta}+\tau^{-2\delta+2\alpha-\epsilon}.
\]
If $\alpha<\delta/2$, we can remove the $\epsilon$-loss:
\[
\sum_{\Aa\in\mathcal{U},\,\Bb\in\mathcal{V}}
\Upsilon_{\mathsf{red}(\Aa\overline{\Bb})}^{\alpha}
\ll_{\alpha,C_0}\tau^{-2\delta+2\alpha}.
\]
\end{lemma}

\begin{proof}
By \cite[Lemma~10]{CalMag25}, the Frobenius norm satisfies $\|\gamma_\Aa\| \asymp \Upsilon_\Aa^{-1/2}$. Using the identity $\|\gamma\|^2 = 2\cosh(d(\mathrm{i},\gamma(\mathrm{i})))$ and the lattice point estimate $\#\{\gamma\in\Gamma: d(\mathrm{i},\gamma(\mathrm{i}))<R\}\ll e^{\delta R}$ from \cite{Pa88} (see also \cite[Theorem~14.22]{Bor16}), it follows that
\begin{equation}\label{eq:cardsets}
\#\left\{ \Aa \in \mathcal{W} : \Upsilon_{\Aa} > \tau \right\} \ll \tau^{-\delta}.
\end{equation}
Thus, using a dyadic decomposition of the set $\{\Aa \in \mathcal{W} : \Upsilon_{\Aa} > \tau\}$, we get for every $\kappa\in\mathbb{R}$ and $\epsilon>0$,
\begin{align*}
\sum_{\substack{\Aa \in \mathcal{W}\\ \Upsilon_{\Aa}>\tau}} \Upsilon_{\Aa}^{\kappa} 
&\ll O(1) + \sum_{k\ge 0: 2^{-k}>\tau} 2^{-\kappa k} \#\left\{ \Aa \in \mathcal{W} : \Upsilon_{\Aa} > 2^{-k-1} \right\}\\
&\ll O(1) + \sum_{k\ge 0: 2^{-k}>\tau} 2^{(\delta-\kappa) k},
\end{align*}
where $O(1)$ accounts for the finitely many terms with $\Upsilon_{\Aa}>1$. A geometric series summation gives
\begin{equation}\label{eq:weighted_Upsilon_sum}
\sum_{\substack{\Aa \in \mathcal{W}\\ \Upsilon_{\Aa}>\tau}} \Upsilon_{\Aa}^{\kappa} \ll_{\kappa,\epsilon} \tau^{-\epsilon}+\tau^{-\delta+\kappa}.
\end{equation}
If $\kappa>\delta$, the same geometric series converges uniformly in $\tau$, so the left-hand side is $O_\kappa(1)$.

Now we can prove the statement of Lemma~\ref{lem:sum_of_Upsilons}. For $\Aa\in\mathcal{U}$ and $\Bb\in\mathcal{V}$, let $\Cc$ be their maximal common suffix and write $\Aa=\Aa_0\Cc$ and $\Bb=\Bb_0\Cc$. If both words $\Aa_0$ and $\Bb_0$ are non-empty, then by maximality of $\Cc$ they end with different letters. In this case $\Aa_0 \overline{\Bb_0}$ is reduced. The same is trivially true when at least one of $\Aa_0$ and $\Bb_0$ is empty. In any case, $\mathsf{red}(\Aa\overline{\Bb}) = \Aa_0 \overline{\Bb_0}$, and using the convention $\Upsilon_\emptyset=1$, Lemma~\ref{lem:boundsderivatives} gives
\[
\Upsilon_{\mathsf{red}(\Aa\overline{\Bb})} \asymp \Upsilon_{\Aa_0}\Upsilon_{\Bb_0},
\]
and
\[
 C^{-1} \frac{\tau}{\Upsilon_\Cc} \le \Upsilon_{\Aa_0}, \Upsilon_{\Bb_0} \le C \frac{\tau}{\Upsilon_\Cc},
\qquad
\Upsilon_\Cc\ge C^{-1}\tau,
\]
for some constant $C=C(C_0,\Gamma)>0$. Therefore, since the map $(\Aa,\Bb)\mapsto(\Aa_0,\Bb_0,\Cc)$ is injective, we can estimate
\begin{equation}
\sum_{\Aa\in\mathcal{U},\,\Bb\in\mathcal{V}} \Upsilon_{\mathsf{red}(\Aa\overline{\Bb})}^{\alpha} \ll_{\alpha,C_0} \tau^{2\alpha}
\sum_{\substack{\Cc\in\mathcal{W}\\\Upsilon_\Cc\ge C^{-1}\tau}}
\#\left\{(\Aa_0,\Bb_0)\in\mathcal{W}^2:
\Upsilon_{\Aa_0},\Upsilon_{\Bb_0}\ge C^{-1} \frac{\tau}{\Upsilon_\Cc}\right\}
\Upsilon_\Cc^{-2\alpha}.
\end{equation}
Using \eqref{eq:cardsets} and \eqref{eq:weighted_Upsilon_sum}, we obtain
\begin{align*}
\sum_{\Aa\in\mathcal{U},\,\Bb\in\mathcal{V}} \Upsilon_{\mathsf{red}(\Aa\overline{\Bb})}^{\alpha} &\ll_{\alpha,C_0} \tau^{-2\delta+2\alpha}
\sum_{\substack{\Cc\in\mathcal{W}\\\Upsilon_\Cc\ge C^{-1}\tau}}
\Upsilon_\Cc^{2\delta-2\alpha}\\
&\ll_{\alpha,\epsilon,C_0}
\tau^{-2\delta+2\alpha}\left(\tau^{-\epsilon}+\tau^{-\delta+(2\delta-2\alpha)}\right)\\
&=\tau^{-\delta}+\tau^{-2\delta+2\alpha-\epsilon},
\end{align*}
which is the sought-after estimate.
If $\alpha<\delta/2$, then $2\delta-2\alpha>\delta$, and the sum over $\Cc$ in the first line is uniformly bounded by the loss-free weighted estimate above (with a fixed adjustment of the cutoff to include equality). This gives $\ll_{\alpha,C_0}\tau^{-2\delta+2\alpha}$ and proves the loss-free case.
\end{proof}

\subsection{Proof of the HS-norm bound}
We can now give the proof of Proposition~\ref{prop:HS_bound}. By Proposition~\ref{prop:approximation}, for every fixed integer $N\ge1$, all sufficiently small $h,\tau>0$, all $f\in H^2(\Omega(h))$, all $s=\sigma+it$, and any $\epsilon>0$, writing $E=e^{C(h|t|+|\sigma|)}$, we have
\[
\begin{aligned}
\Vert \mathcal L_{\tau,s}f \Vert_{L^2(\Omega(h))}^2
&\ll_{N,\epsilon} E\tau^{-2\delta+2\sigma-\epsilon}h^{2-\delta}\int_\Lambda |f|^2\,d\mu_\delta\\
&\quad+E\tau^{N+2\sigma}h^{-\delta}\|f\|_{L^2(\Omega(h))}^2.
\end{aligned}
\]
Fix an orthonormal basis $\{\psi_m\}$ of $H^2(\Omega(h))$. Applying the estimate to $f=\mathcal L_{\tau',\tau'';s}\psi_m$ and summing over $m$ gives
\[
\begin{aligned}
\|\mathcal L_{\tau,\tau',\tau'';s}\|_{\mathrm{HS}}^2
&\ll_{N,\epsilon} E\tau^{-2\delta+2\sigma-\epsilon}h^{2-\delta}
\sum_m\int_\Lambda|\mathcal L_{\tau',\tau'';s}\psi_m|^2\,d\mu_\delta\\
&\quad+E\tau^{N+2\sigma}h^{-\delta}\|\mathcal L_{\tau',\tau'';s}\|_{\mathrm{HS}}^2.
\end{aligned}
\]
First, we need to estimate the HS-norm on the right-hand side. A calculation similar to that in the proof of Proposition~\ref{prop:integral_bound} gives
\begin{equation}
\Vert \mathcal{L}_{\tau', \tau''; s} \Vert_{\mathrm{HS}}^2
=\sum_{b\in\mathcal A} \sum_{(\Aa_1,\Aa_2),(\Bb_1,\Bb_2)\in S_b} \int_{\Omega_b(h)} g_{\Aa_1\Aa_2,\Bb_1\Bb_2}(z;s)\,\dvol(z).
\end{equation}
Using \eqref{eq:pointwise_g_bound}, Lemma~\ref{lem:YZ}, and Lemma~\ref{lem:vol_bound_omega}, we get
\begin{align*}
\Vert \mathcal{L}_{\tau', \tau''; s} \Vert_{\mathrm{HS}}^2 &\ll E h^{-2} (\tau'\tau'')^{2\sigma} |Y(\tau')|^2|Y(\tau'')|^2 \vol(\Omega(h))\\
&\ll E h^{-\delta} (\tau'\tau'')^{-2\delta + 2\sigma}.
\end{align*}
Now, we set 
$$
\tau'' := C_{\mathrm{sep}} h, \quad \tau' := T^{-\beta_{\mathrm{F}}/(2\delta)}.
$$
Proposition~\ref{prop:integral_bound} gives for all $s = \sigma + it \in \mathbb{C}$ with $T/2 \le |t| \le 2T$, $T$ sufficiently large,
\[
\sum_{m\in \mathbb{N}} \int_{\Lambda} |\mathcal{L}_{\tau', \tau''; s} \psi_m|^2 \, d\mu_\delta \ll e^{C|\sigma|} h^{-2} (\tau'\tau'')^{-\delta + 2\sigma}.
\]
Moreover, we choose
$$
h = \frac{1}{T},
$$
which gives $E\ll e^{C|\sigma|}$, thereby eliminating the exponential blow-up in the variable $t$. Inserting the last two estimates into the bound for $\|\mathcal L_{\tau,\tau',\tau'';s}\|_{\mathrm{HS}}^2$ gives
\begin{align*}
\| \mathcal{L}_{\tau, \tau', \tau''; s}\|_{\mathrm{HS}}^2 &\ll_{N,\epsilon} e^{C |\sigma|} \left( h^{-\delta} \tau^{-2\delta+2\sigma-\epsilon}  (\tau'\tau'')^{-\delta + 2\sigma} + h^{-2\delta} \tau^{N + 2\sigma} ( \tau'\tau'')^{-2\delta + 2\sigma}  \right)\\
&\ll_{N,\epsilon} e^{C |\sigma|} h^{-\delta} (\tau'\tau'')^{-\delta + 2\sigma} \left( \tau^{-2\delta+2\sigma-\epsilon}  + h^{-\delta} \tau^{N + 2\sigma} ( \tau'\tau'')^{-\delta}  \right)\\
 &\ll_{N,\epsilon} e^{C |\sigma|} T^{A(\delta - \sigma) - \frac{\beta_{\mathrm{F}}}{2}} \left( \tau^{-2\delta + 2\sigma-\epsilon}  + T^{A'} \tau^{N + 2\sigma} \right),
\end{align*}
where $A=2+\beta_{\mathrm{F}}/\delta$ and $A'=2\delta+\beta_{\mathrm{F}}/2$. Choose
$$
\tau = c T^{-\frac{A'}{2\delta + N}}\leq c
$$
for a small enough constant $c > 0$, so that $\tau$ is small enough for the above estimates to apply. Since $\sigma>0$, inserting this choice into the previous bound gives
$$
\| \mathcal{L}_{\tau, \tau', \tau''; s}\|_{\mathrm{HS}}^2 \ll_{N,\epsilon} e^{C \sigma} T^{A(\delta - \sigma) - \frac{\beta_{\mathrm{F}}}{2} + \frac{A'(2\delta+\epsilon)}{N+2\delta}}.
$$
Choose a fixed $N$, depending only on $\epsilon$ and $\delta$, large enough that $A'(2\delta+\epsilon)/(N+2\delta)\le\epsilon$. All four parameters depend only on $T$, $\epsilon$, and $\Gamma$, they tend to $0$ as $T\to\infty$, and satisfy $\tau\tau'\tau''\le cC_{\mathrm{sep}}T^{-1-\beta_{\mathrm{F}}/(2\delta)}\le T^{-1}$ for large $T$. This proves Proposition~\ref{prop:HS_bound}.

\normalem
\bibliography{improved}

@article{JN2016,
  author    = {Jakobson, Dmitry and Naud, Fr{\'e}d{\'e}ric},
  title     = {{Resonances and density bounds for convex co-compact congruence subgroups of ${SL}_2(\mathbb{Z})$}},
  journal   = {Israel Journal of Mathematics},
  volume    = {213},
  number    = {1},
  pages     = {443--473},
  year      = {2016},
  publisher = {Springer},
  doi       = {10.1007/s11856-016-1332-7},
  url       = {https://doi.org/10.1007/s11856-016-1332-7}
}

@article{JN2012,
  author    = {Jakobson, Dmitry and Naud, Fr{\'e}d{\'e}ric},
  title     = {{On the Critical Line of Convex Co-Compact Hyperbolic Surfaces}},
  journal   = {Geometric and Functional Analysis},
  volume    = {22},
  number    = {2},
  pages     = {352--368},
  year      = {2012},
  publisher = {Springer},
  doi       = {10.1007/s00039-012-0154-y},
  url       = {https://doi.org/10.1007/s00039-012-0154-y}
}

@article{Remez1936,
  author = {E. J. Remez},
  title = {Sur une propri\'{e}t\'{e} des polyn\^{o}mes de {Tchebycheff}},
  journal = {Comm. Inst. Sci. Kharkov},
  volume = {13},
  pages = {93--95},
  year = {1936}
}

@article{BD18,
  author = {Jean Bourgain and Semyon Dyatlov},
  title = {Spectral gaps without the pressure condition},
  journal = {Annals of Mathematics},
  volume = {187},
  number = {3},
  pages = {825--867},
  year = {2018}
}

@article{Bourgain2010,
  author = {Jean Bourgain},
  title = {The discretized sum-product and projection theorems},
  journal = {J. Anal. Math.},
  fjournal = {Journal d'Analyse Math\'{e}matique},
  volume = {112},
  number = {1},
  pages = {193--236},
  year = {2010},
  doi = {10.1007/s11854-010-0028-x}
}

@misc{LequenSahlsten,
  author = {F\'{e}lix Lequen and Tuomas Sahlsten},
  title = {Quantitative {Fourier} decay for {Patterson--Sullivan} measures of dimension larger than $1/2$},
  year = {2026},
  eprint = {2607.18010},
  archivePrefix = {arXiv},
  primaryClass = {math.DS},
  url = {https://arxiv.org/abs/2607.18010},
  note = {arXiv:2607.18010}
}

@article{DZ16,
  author = {Semyon Dyatlov and Joshua Zahl},
  title = {Spectral gaps, additive energy, and a fractal uncertainty principle},
  journal = {Geom. Funct. Anal.},
  fjournal = {Geometric and Functional Analysis},
  volume = {26},
  number = {5},
  pages = {1011--1094},
  year = {2016}
}

@article{Pa76,
  author = {Samuel J. Patterson},
  title = {The limit set of a {F}uchsian group},
  journal = {Acta Math.},
  fjournal = {Acta Mathematica},
  volume = {136},
  number = {3-4},
  pages = {241--273},
  year = {1976},
  doi = {10.1007/BF02392052}
}

@article{Sul79,
  author = {Dennis P. Sullivan},
  title = {The density at infinity of a discrete group of hyperbolic motions},
  journal = {Publ. Math. Inst. Hautes \'{E}tudes Sci.},
  fjournal = {Publications Mathématiques de l'Institut des Hautes Études Scientifiques},
  volume = {50},
  pages = {171--202},
  year = {1979},
  doi = {10.1007/BF02684773}
}

@article{Pa88,
  author = {Samuel J. Patterson},
  title = {On a lattice-point problem in hyperbolic space and related questions in spectral theory},
  journal = {Ark. Mat.},
  volume = {26},
  pages = {167--172},
  year = {1988}
}

@article{JNS20,
  author = {Dmitry Jakobson and Fr\'{e}d\'{e}ric Naud and Louis Soares},
  title = {Large degree covers and sharp resonances of hyperbolic surfaces},
  journal = {Ann. Inst. Fourier},
  fjournal = {Annales de l'Institut Fourier},
  volume = {70},
  number = {2},
  year = {2020},
  pages = {523--596}
}

@book{Bowen1975,
  author = {Rufus Bowen},
  title = {Equilibrium States and the Ergodic Theory of Anosov Diffeomorphisms},
  series = {Lecture Notes in Mathematics},
  volume = {470},
  publisher = {Springer-Verlag, Berlin},
  year = {1975}
}

@article {Naud2,
    AUTHOR = {Fr\'{e}d\'{e}ric {Naud}},
     TITLE = {Expanding maps on {C}antor sets and analytic continuation of
              zeta functions},
   JOURNAL = {Ann. Sci. \'{E}cole Norm. Sup. (4)},
  FJOURNAL = {Annales Scientifiques de l'\'{E}cole Normale Sup\'{e}rieure. Quatri\`eme
              S\'{e}rie},
    VOLUME = {38},
      YEAR = {2005},
    NUMBER = {1},
     PAGES = {116--153},
      ISSN = {0012-9593},
   MRCLASS = {37C30 (37A45 37B10 37C35 37D35 37D40 37F50)},
  MRNUMBER = {2136484},
MRREVIEWER = {Dieter H. Mayer},
       DOI = {10.1016/j.ansens.2004.11.002},
       URL = {https://doi.org/10.1016/j.ansens.2004.11.002},
}

@article{BJP03,
  author = {David Borthwick and Christopher Judge and Peter A. Perry},
  title = {Determinants of {L}aplacians and isopolar metrics on surfaces of infinite area},
  journal = {Duke Math. J.},
  fjournal = {Duke Mathematical Journal},
  volume = {118},
  pages = {61--102},
  year = {2003}
}

@article{PaPe01,
  author = {Samuel J. Patterson and Peter Perry},
  title = {The divisor of {S}elberg's zeta function for {K}leinian groups},
  journal = {Duke Math. J.},
  fjournal = {Duke Mathematical Journal},
  volume = {106},
  pages = {321--390},
  year = {2001},
  note = {Appendix A by Charles Epstein}
}

@misc{Cun26,
  title         = {Improved fractal Weyl bounds matching improved spectral gaps for hyperbolic surfaces and open quantum maps},
  author        = {Travis Cunningham},
  year          = {2026},
  eprint        = {2602.20392},
  archivePrefix = {arXiv},
  primaryClass  = {math.SP},
  url           = {https://arxiv.org/abs/2602.20392},
  note          = {arXiv:2602.20392}
}

@misc{Soa25,
  title         = {Improved fractal Weyl bounds for convex cocompact hyperbolic surfaces and large resonance-free regions},
  author        = {Louis Soares},
  year          = {2025},
  eprint        = {2301.03023},
  archivePrefix = {arXiv},
  primaryClass  = {math.SP},
  url           = {https://arxiv.org/abs/2301.03023},
  note          = {arXiv:2301.03023}
}

@article{Liverani1995,
  author    = {Carlangelo Liverani},
  title     = {Decay of correlations},
  journal   = {Annals of Mathematics},
  volume    = {142},
  number    = {2},
  pages     = {239--301},
  year      = {1995},
  doi       = {10.2307/2118636},
  publisher = {Princeton University Press},
}

@book{ParryPollicott1990,
  author    = {William Parry and Mark Pollicott},
  title     = {Zeta Functions and the Periodic Orbit Structure of Hyperbolic Dynamics},
  series    = {Astérisque},
  volume    = {187--188},
  year      = {1990},
  publisher = {Société Mathématique de France},
  pages     = {268},
}

@article {BowenSeries1,
  AUTHOR = {Rufus Bowen and Caroline Series},
   TITLE = {Markov maps associated with {F}uchsian groups},
   JOURNAL = {Inst. Hautes \'{E}tudes Sci. Publ. Math.},
  FJOURNAL = {Institut des Hautes \'{E}tudes Scientifiques. Publications Math\'{e}matiques},
  NUMBER = {50},
    YEAR = {1979},
   PAGES = {153--170},
    ISSN = {0073-8301},
   MRCLASS = {58F11 (30C60 30F35)},
  MRNUMBER = {556585},
MRREVIEWER = {Troels J\o rgensen},
     URL = {http://www.numdam.org/item?id=PMIHES_1979__50__153_0},
}

@article {Bowen1,
    AUTHOR = {Rufus Bowen},
     TITLE = {Hausdorff dimension of quasicircles},
   JOURNAL = {Inst. Hautes \'{E}tudes Sci. Publ. Math.},
  FJOURNAL = {Institut des Hautes \'{E}tudes Scientifiques. Publications
              Math\'{e}matiques},
    NUMBER = {50},
      YEAR = {1979},
     PAGES = {11--25},
      ISSN = {0073-8301},
   MRCLASS = {57S05 (30C20 32G15 58F11)},
  MRNUMBER = {556580},
MRREVIEWER = {L. A. Bunimovich},
       URL = {http://www.numdam.org/item?id=PMIHES_1979__50__11_0},
}

@article{CalMag25,
  author  = {Irving Calder\'on and Michael Magee},
  title   = {Explicit spectral gap for Schottky subgroups of $\mathrm{SL}(2,\mathbb{Z})$},
  journal = {J. Eur. Math. Soc.},
  year    = {2025},
  doi     = {10.4171/JEMS/1610},
  url     = {https://doi.org/10.4171/JEMS/1610},
  note    = {Published online}
}

@article{Lu_Sridhar_Zworski,
  author  = {Wen-Te Lu and Srinivas Sridhar and Maciej Zworski},
  title   = {Fractal Weyl Laws for Chaotic Open Systems},
  journal = {Phys. Rev. Lett.},
  volume  = {91},
  number  = {15},
  pages   = {154101},
  year    = {2003},
  month   = {oct},
  doi     = {10.1103/PhysRevLett.91.154101},
  url     = {https://doi.org/10.1103/PhysRevLett.91.154101}
}

@article{DZ20,
   title = {Fractal uncertainty for transfer operators},
   author = {Semyon Dyatlov and Maciej Zworski},
   journal = {Int. Math. Res. Not. IMRN},
   fjournal = {International Mathematics Research Notices},
   year = {2020},
   volume = {2020},
   number = {3},
   pages = {781--812},
   doi = {10.1093/imrn/rny026}
}

@book{Krantz,
   title = {Geometric analysis of the {B}ergman kernel and metric},
   author = {Steven G. Krantz},
   series = {Graduate Texts in Mathematics},
   volume = {268},
   year = {2013},
   publisher = {Springer, New York},
   pages = {xiii+292},
   isbn = {978-1-4614-7923-9},
   doi = {10.1007/978-1-4614-7924-6}
}

@article{Dya19,
   title = {Improved fractal {W}eyl bounds for hyperbolic manifolds},
   author = {Semyon Dyatlov},
   note = {With an appendix by David Borthwick, Semyon Dyatlov, and Tobias Weich},
   journal = {J. Eur. Math. Soc.},
   fjournal = {Journal of the European Mathematical Society},
   year = {2019},
   volume = {21},
   number = {6},
   pages = {1595--1639},
   doi = {10.4171/JEMS/867}
}

@article{BD17,
    author = {Jean Bourgain and Semyon Dyatlov},
    title = {Fourier dimension and spectral gaps for hyperbolic surfaces},
    journal = {Geom. Funct. Anal.},
    fjournal = {Geometric and Functional Analysis},
    volume = {27},
    pages = {744--771},
    number = {4},
    year = {2017},
    doi = {10.1007/s00039-017-0412-0}
}

@book {Simon,
    AUTHOR = {Barry Simon},
     TITLE = {{Trace ideals and their applications}},
    SERIES = {Mathematical Surveys and Monographs},
    VOLUME = {120},
   EDITION = {Second},
 PUBLISHER = {American Mathematical Society, Providence, RI},
      YEAR = {2005},
     PAGES = {viii+150},
      ISBN = {0-8218-3581-5},
   MRCLASS = {47L20 (47A40 47A55 47B10 47B36 47E05 81Q15 81U99)},
  MRNUMBER = {2154153},
MRREVIEWER = {Pavel B. Kurasov},
}

@incollection{Button,
    AUTHOR = {Jack O. Button},
     TITLE = {All {F}uchsian {S}chottky groups are classical {S}chottky
              groups},
 BOOKTITLE = {The {E}pstein birthday schrift},
    SERIES = {Geom. Topol. Monogr.},
    VOLUME = {1},
     PAGES = {117--125},
 PUBLISHER = {Geom. Topol. Publ., Coventry},
      YEAR = {1998},
   MRCLASS = {20H10 (30F35)},
  MRNUMBER = {1668339},
MRREVIEWER = {C. Maclachlan},
       DOI = {10.2140/gtm.1998.1.117},
       URL = {https://doi.org/10.2140/gtm.1998.1.117},
}

@article {Naud14,
    AUTHOR = {Fr\'{e}d\'{e}ric Naud},
     TITLE = {Density and location of resonances for convex cocompact hyperbolic surfaces},
   JOURNAL = {Invent. Math.},
  FJOURNAL = {Inventiones Mathematicae},
    VOLUME = {195},
      YEAR = {2014},
    NUMBER = {3},
     PAGES = {723--750},
      ISSN = {0020-9910},
   MRCLASS = {58J50 (30F30 35P20 35R01)},
  MRNUMBER = {3166217},
MRREVIEWER = {Emil Saucan},
       DOI = {10.1007/s00222-013-0463-2},
       URL = {https://doi.org/10.1007/s00222-013-0463-2}
}

@book{Bor16,
    author = {David Borthwick},
    title = {Spectral theory of infinite-area hyperbolic surfaces},
    series = {Progress in Mathematics},
    volume = {318},
    edition = {2nd},
    pages = {xiii+463},
    year = {2016},
    publisher = {Birkh\"{a}user/Springer, Cham},
    isbn = {978-3-319-33877-4},
    doi = {10.1007/978-3-319-33878-1}
}

@Article{GLZ,
    Author = {Laurent Guillop\'{e} and Kevin Lin and Maciej Zworski},
    Title = {{The Selberg zeta function for convex cocompact Schottky groups}},
    FJournal = {{Communications in Mathematical Physics}},
    Journal = {{Commun. Math. Phys.}},
    ISSN = {0010-3616; 1432-0916/e},
    Volume = {245},
    Number = {1},
    Pages = {149--176},
    Year = {2004},
    _Publisher = {Springer-Verlag, Berlin},
    _Language = {English},
    _DOI = {10.1007/s00220-003-1007-1},
    _MSC2010 = {11M36 58J52},
    _Zbl = {1075.11059}
}

@Article{NM20,
    Author = {Michael Magee and Fr\'{e}d\'{e}ric Naud},
    Title = {Explicit spectral gaps for random covers of {R}iemann surfaces},
    FJournal = {Publications math\'{e}matiques de l'IH\'{E}S},
    Journal = {{Publ. math. IHES}},
    Volume = {132},
    Pages = {137--179},
    Year = {2020}
}
\bibliographystyle{amsplain}

\end{document}